\documentclass[12pt]{iopart}
\usepackage{amsfonts}
\usepackage{epsfig}
\usepackage{amssymb,amsthm,amscd}
\usepackage{latexsym}
\usepackage{xcolor,bm}
\usepackage{cite}
\usepackage{color}
 \theoremstyle{theorem}
 \newtheorem{thm}{Theorem}[section]
 \newtheorem{lem}{Lemma}[section]
  
\def\C{\mathcal{C}}
\def\T{\mathcal{T}}
\def\tn{{\tilde{n}}}
\def\tJ{{\tilde{J}}}
\def\tE{{\tilde{\mathcal{E}}}}
\def\tV{{\tilde{V}}}
\begin{document}

\title[A full viscous quantum hydrodynamic system]{Stability of a one-dimensional full viscous quantum hydrodynamic system}

\author{Xiaoying Han}

\address{221 Parker Hall, Department of Mathematics and Statistics \\ Auburn University, Auburn, AL 36849 USA}
\ead{xzh0003@auburn.edu}

\author{Yuming Qin}
\address{Department of Mathematics, Institute for Nonlinear Science \\ Donghua University, Shanghai, 201620, China}
\ead{yuming\_qin@hotmail.com}

\author{Wenlong Sun}
\address{School of information and Mathematics \\Yangtze University, Jingzhou, Hubei 434100, China}
\ead{wlsun@yangtzeu.edu.cn}

\vspace{10pt}
\begin{indented}
\item[]March 2023
\end{indented}

\begin{abstract}
A full
viscous quantum hydrodynamic system for particle density, current
density, energy density and electrostatic potential coupled with a
Poisson equation in one dimensional bounded intervals is studied.  First,  the existence and uniqueness of a steady-state solution to
the quantum hydrodynamic system is established.  Then, utilizing the fact that the third
order perturbation term has an appropriate sign,  the
local-in-time existence of the solution is investigated by introducing a fourth
order viscous regularization and using the entropy dissipation
method.   In the end, the exponential
stability of the steady-state solution is shown by constructing a uniform
a-priori estimate.
\end{abstract}

\vspace{2pc} \noindent{\it Keywords}: viscous quantum hydrodynamic
system, exponential stability, uniform estimate, local-in-time
solution, entropy dissipation method.

%
%
%

\section{Introduction}

The quantum hydrodynamic (QHD)  model  has been widely used to study semiconductor
devices of ultra small size, e.g.,  nano-size, where the effects of
quantum mechanics arise.   In particular, they play a crucial role in
describing the motion of electron or hole transport under self-consistent electric fields.  The advantage of macroscopic
QHDs lies in that they provide  descriptions of
the dynamic evolution of physical observables  and
thus can be utilized to simulate quantum phenomena directly.   Moreover, in the
semiclassical, i.e.,  the zero dispersion limit, the macroscopic
quantum quantities such as density, momentum, and temperature, are
shown to converge in some sense to the Newtonian fluid-dynamical
quantities \cite{GLM00}.        Macroscopic quantum models are also
used in other physical area such as superfluid \cite{LM93} and
superconductivity \cite{F72}.

This work is based on a full viscous QHD that can be derived from the Wigner equation with the Fokker-Plank
collision operator \cite{JM07}.   More precisely, consider  the following
equations on a bounded domain $\Omega \subset \mathbb{R}^3$:
\begin{eqnarray}\label{1.1}
\hspace{-1in} \left\{
  \begin{array}{ll}
  \displaystyle \partial_t n + \nabla \cdot J = \nu \Delta n, \medskip \\
   \displaystyle \partial_t J + \nabla \cdot \left( \frac{J\otimes J}{n} \right)
    +  \nabla \cdot P
    - n \nabla V
        = - \frac{J}{\tau}
     + \nu\Delta J
     - \mu\nabla n, \medskip \\
   \displaystyle \partial_t \mathcal{E}
    + \nabla \cdot \left( \frac{J}{n}( \mathcal{E}+P) \right)
    - J\cdot \nabla V
   = - \frac{2}{\tau}\left( \mathcal{E}-\frac32 n \right)
     + \nu\Delta \mathcal{E}
     - \mu \nabla \cdot J,
  \end{array}
 \right.
\end{eqnarray}
where $n>0, J$ and $ \mathcal{E}$ denote the particle density, the current
density and the energy density, respectively.    The
 parameters $\nu > 0$, $\epsilon > 0$, $\tau > 0$ and $\mu$  stand for the scaled
viscosity constant, the Planck constant, the relaxation time, and the
interaction constant, respectively, and the term $J\otimes J$ denotes the matrix with
components $J_jJ_k$.

Here the stress tensor $P$ and energy density $ \mathcal{E}$ can be
expressed as
$$
  P = nTI - \frac{\epsilon^2}{12}n(\nabla\otimes\nabla)\log n, \quad
 \mathcal{E}= \frac{|J|^2}{2n} + \frac32 nT
      - \frac{\epsilon^2}{24}n\Delta\log n,
$$
where $T$ is the particle temperature and $I$ denotes the
identity matrix.   The function $V=V(x,t)$ represents the
electric potential that is self-consistently coupled with the Poisson equation $$ \lambda^2 \Delta V
    = n - \rho(x),$$ where $\lambda$ is the scaled Debye length and $\rho(x)$ is the doping profile modeling
the semiconductor device under consideration \cite{J01,MRS90}.

It is worth mentioning   that the viscous terms $\nu \Delta n, \nu\Delta J$ and
$\nu\Delta \mathcal{E}$ on the right hand sides in the system (\ref{1.1}) are formally derived from the Wigner-Fokker-Planck
model.  They are not ad-hoc regularization of the QHD model. The above viscous regularization is different
from the viscous terms in the classical Navier-Stokes equations
since it models the interactions of electrons and phonons in a
semiconductor crystal.   One basic character of the QHD model
is that the energy density equation consists of an additional
quantum correction term of the order $\mathcal{O}(\epsilon)$ introduced first
by Wigner \cite{W32} in 1932.    Besides,  the stress tensor contains
an additional quantum correction part \cite{AT87,AI89} related
to the quantum Bohm potential (or internal self-potential)
\cite{B52} $$\displaystyle Q(n) =
\frac{\epsilon^2}{2}\frac{\Delta\sqrt{n}}{\sqrt{n}}\quad \mbox{ through the
formula}\quad \displaystyle \nabla Q = \frac{\epsilon^2}{4n}\nabla\cdot
(n(\nabla\otimes\nabla){\rm log}n).$$

Many efforts have been contributed to the mathematical analysis for
the QHD model in recent years,   on the steady-state solutions as
well as  evolutionary solutions.   Most studies concern
 isentropic non-viscous QHD models in one
dimensional or multidimensional bounded domain with various types of
boundary conditions.   In particular, for one dimensional
evolutionary QHD model, the local-in-time existence of solutions
\cite{HJL03} and global existence theory with the exponential
stability of stationary state in whole space \cite{HLM06,JL04} were
established.  For multidimensional evolutional QHD model, the local
existence theory and exponential stability of equilibrium state
analysis were also investigated for ir-rotational fluid on spatial
periodic domain in \cite{LM04}. The corresponding existence theory
for time-dependent solution for general rotational fluid is usually
difficult and was obtained in \cite{HLMO10}, where the exponential
decay to the stationary state obtained therein was made. Moreover,
the asymptotical small scaling analysis including the relaxation
time limit and small Debye length limit for the global solutions
were studied in \cite{JLM06} and \cite{LL05}, respectively.   In
\cite{AM09}, Antonelli and Marcati proved the global existence of
the finite energy weak solutions. They also established the
existence of the global weak solution in two dimensional case in
\cite{AM12}. For viscous isentropic QHD model, based on the entropy
dissipation method, the exponential decay toward a constant steady
state in one dimensional setting with a certain boundary condition
was proved in \cite{GJT03}.  In the multidimensional case, the
isentropic viscous model was investigated in \cite{CD11,D08}, where
the authors proved the local existence of smooth solutions and
exponential decay to the thermal equilibrium state by the entropy
dissipation method, also they obtained the inviscid limit.
 {\color{blue} More
recently, the references \cite{HZ} and \cite{RH21} studied
non-isentropic non-viscous quantum hydrodynamic system for
semiconductors in one-dimensional bounded domain and $\mathbb{R}^3$,
respectively. Particularly, the authors in \cite{HZ} investigated
the existence and asymptotic stability of a stationary solution by
Leray-Schauder fixed point theorem, the refined energy method and
the iteration method with the energy estimates, and the authors in
\cite{RH21} proved the existence, uniqueness and exponential decay
of strong solution near constant states by using energy method when
the coefficients meet certain conditions. Here we want to mention
that due to the appearance of the three order quantum terms, the a
priori estimates of solutions to non-isentropic non-viscous quantum
hydrodynamic system cannot be closed in space
$\mathbb{H}^m(\mathbb{R}^d)$ $(m \geqslant 0)$, so the authors in
\cite{RH21} took the {\rm curl} of the momentum conservation
equation and established a prior estimates of the solutions using
the fact ${\rm rot} \nabla \cdot =0$ and the curl-div decomposition
of the gradient, which makes the absence of the $L^2$-estimate of
velocity filed in the a priori estimates. }

{\color{blue}

However, to the best of our knowledge, there is no analysis in the
literature on mathematical analysis  of the {\it full viscous} QHD
model for semiconductors except Reference \cite{SLH23}, in which the
authors investigated the existence and exponential stability of the
stationary solution for the full viscous QHD model on the real set
$\mathbb{R}$ by the standard continuation argument based on the
local existence and the uniform a priori estimate. Since the
techniques for classical hydrodynamic equations are not applicable
due to the quantum term, the existence of a local-in-time solution
is obtained in \cite{SLH23} by showing the existence of
local-in-time solutions of a reformulated system via the iteration
method. }

In this paper, we consider the full viscous QHD model in a one
dimensional bounded interval and investigate the existence,
uniqueness and the stability of steady-state solutions. Without loss
of generality, we consider the following
 one dimensional  scaled system on the
bounded interval $[0,1]$ derived from (\ref{1.1})
\begin{eqnarray}\label{1.2}
 \hspace{-0.8in} \left \{
  \begin{array}{ll}
   n_t + J_x - \nu n_{xx} = 0,   \medskip\\
  \displaystyle
   J_t + \frac{J}{\tau} - \nu J_{xx}
    + \frac23 \left(\frac{J^2}{n}\right)_x
    + \frac23 \mathcal{E}_x
    - \frac{\epsilon^2}{9}n \left(\frac{(\sqrt{n})_{xx}}{\sqrt{n}}\right)_x
    - nV_x + \mu n_x
   = 0,    \medskip\\
  \displaystyle
   \mathcal{E}_t + \frac53 \left(\frac{J}{n}\mathcal{E}\right)_x
    + \frac2{\tau}\mathcal{E}
    - \nu\mathcal{E}_{xx}
    - \frac13\left(\frac{J^3}{n^2}\right)_x
    - \frac{\epsilon^2}{18}\left(\frac{Jn_{xx}}{n}-\frac{Jn_x^2}{n^2}\right)_x
     \medskip \\
    \qquad\qquad\qquad\qquad\qquad\qquad\qquad\qquad\quad
    \displaystyle - JV_x+ \mu J_x   - \frac{3}{\tau}n
   = 0,  \medskip \\
  \displaystyle
   \lambda^2 V_{xx} = n - \rho(x),
  \end{array}
 \right.
\end{eqnarray}
equipped with   the boundary conditions
\begin{eqnarray}\label{1.4}
  \hspace{-0.8in} \left\{
  \begin{array}{ll}
   \partial_x^k n(0) = \partial_x^k n(1)
    = \partial_x^k {\rho}(0) = \partial_x^k{\rho}(1),\ \ k=0,1,2,3,
    \medskip \\
   \displaystyle J(0) = J(1) = J_b, \ \
   \mathcal{E}(0) = \mathcal{E}(1) = \frac32\rho(0),  \medskip \\
   V(0)=0, \ \ V(1)=V_b, \ \ V_x(0)=V_x(1)=0
  \end{array}
 \right.
\end{eqnarray}
and the
initial  conditions
\begin{eqnarray}\label{1.3}
  \hspace{-1in} n(x,0)=n_0(x)>0, \  J(x,0)=J_0(x), \
 \mathcal{E}(x,0)=\mathcal{E}_0(x),  \ V(x,0)=V_0(x),  \ x \in
 [0,1].
\end{eqnarray}

{\color{blue} The major difficulty in studying the full QHD model is
the dispersive and complex structure of the model.  Another
difficulty is due to the Bohm potential term $Q(n)$, that introduces
a third order perturbation to the Euler-Poisson system,  for which
maximum principles and related tools become inapplicable.   There
have been several attempts in the literature to deal with the
quantum term. Integrating the stationary momentum equation leads to
a second order differential equation to which maximum principle
arguments can be applied \cite{GJ01}.  A fourth order wave equation
is obtained after differentiating the equation with respect to the
spatial variable. This method was employed in \cite{LM04,HLM06} to
prove the existence of global solutions to the isentropic quantum
hydrodynamic model with $\nu=0$, but only for initial data close to
thermal equilibrium. A wave function polar decomposition method was
used in \cite{AM09} to investigate the global existence of solutions
to the inviscid isentropic model with nonnegative particle density.
In \cite{GJV09}, the Faedo-Glerkin method was applied to show the
existence of solution of the viscous isentropic system for $\epsilon
> 0$. The idea of \cite{CD07} was to introduce a bi-Laplacian
regularization in the viscous isentropic model and to employ energy
estimates to establish local existence of solution.

Different from the above works, the idea of the proof in this work
is summarized as follows. To prove the existence of the steady-state
solution, i.e., Theorem \ref{T1.1}, we first linearized the system
(\ref{1.6}) by taking full advantage of the viscous term, the sign
of the third term and the structure of the system.  Then we set up
a-priori estimates by using energy method and finally, we
established the existence of solution via continuity argument.  To
show the local-in-time existence of the solution, i.e., Theorem
\ref{T1.3}, based on the observation that the third order term has
an appropriate sign, we introduced a fourth order viscous
regularization and employed energy estimates to obtain the local
existence of solution. In the end, to investigate the long time
stability of the steady-state solution, i.e., Theorem \ref{T1.2}, we
mainly focus on establishing an upper bound for the solutions. }

\subsection{Notations and standing assumptions}
Denote by $L^p(\mathbb{R})$ and $W^{m,p}(\mathbb{R})$ the usual
Lebesgue space and Sobolev space endowed with norms
$\|\cdot\|_{L^p}$ and $\|\cdot\|_{m,p}$, respectively. For example,
$$
 \|\varphi\|_{L^p}
 := \left(\int_{\mathbb{R}} |\varphi|^p{\mathrm d}x \right)^{1/p}
  = \left(\int |\varphi|^p{\mathrm d}x \right)^{1/p},
$$
and
$$
 \|\varphi\|_{m,p}
 := \left(\sum\limits_{|\beta|\leqslant m}
    \int_{\mathbb R} |D^\beta \varphi|^p{\mathrm d}x \right)^{1/p}
  = \left(\sum\limits_{|\beta|\leqslant m} \int |D^\beta \varphi|^p{\mathrm d}x \right)^{1/p}.
$$
In particular, we denote $H^m(\mathbb{R}):=W^{m,2}(\mathbb{R})$ with
the norm $\|\cdot\|_m$.  Based on the above notations, we further
denote by $L^p(I; X)$ the
 space of strongly measurable functions on
the closed interval $I$, with values in a Banach space $X$,
 endowed with norm
$$\|\varphi\|_{L^p(I; X)}
  := \left(\int_I \|\varphi\|^p_X {\mathrm d}t \right)^{1/p},\ \ \mbox{for }1\leqslant
p<\infty,$$
and denote by $\C(I; X)$
 the space of continuous functions on the interval $I$, with values
 in the Banach space $X$, endowed with
  the usual norm.   Throughout this paper,
the notations $C$ and $C_i$ for $i=1,2,\cdots$ are assigned to
denote positive generic constants.

In this work the doping profile $\rho(x)$ is assumed to satisfy the
following compatibility conditions:
\begin{itemize}
\item[\textbf{(A1)}] \enskip
 $\rho(0)=\rho(1)=\rho_b, \quad
 0 < \rho(x)\in \C^4([0,1]),
 \quad \delta_0 := \|\rho_x\|_3.$
\end{itemize}

The following technical assumption is required to
establish the existence and uniqueness of a steady-state solution to the boundary value problem
  (\ref{1.2}) -- (\ref{1.4}).
\begin{itemize}
\item[\textbf{(A2)}] \enskip
There exists a constant $\alpha_0 \, \in (0,1)$, such that
$$
\frac{1}{\alpha_0}
 \ < \
  \frac{\sqrt{2}\tau\nu(2\mu+5)\rho(x)}{3\lambda^2}, \quad x \in [0,1].
$$
\end{itemize}

\subsection{Main results}
There are three main results of this work, the existence and stability of the steady-state solution of the boundary value problem
  (\ref{1.2}) -- (\ref{1.4}) stated in Theorem \ref{T1.1} and \ref{T1.2} below, respectively,  and  the local-in-time existence of the
solution to the intial-boundary value problem   (\ref{1.2}) -- (\ref{1.3}) stated in Theorem \ref{T1.3} below.
\begin{thm}\label{T1.1}
Let Assumptions  $\mathbf{(A1)}-\mathbf{(A2)}$   hold.  Then  the
BVP (\ref{1.2}) -- (\ref{1.4}) possesses a unique steady-state
solution $(n^*(x), J^*(x), \mathcal{E}^*(x), V^*(x))$ provided $J_b$
and $\delta_0$ sufficiently small. Moreover, the steady-state
solutions satisfies
\begin{eqnarray}\label{1.8}
 \|n^*-\rho\|_3^2 + \|J^*-J_b\|_2^2
  + \|\mathcal{E}^*-\frac32\rho\|_2^2
  + \|V^*\|_5^2
 \leqslant
C \delta_0^2,
\end{eqnarray}
where $C$ is a positive constant independent of $\epsilon$
and $\delta_0$.
\end{thm}

To investigate  the global-in-time stability
of the steady-state, define
$$d_n=n_0-n^*, \ d_J=J_0-J^*, \
d_\mathcal{E}=\mathcal{E}_0-\mathcal{E}^*, \ d_V=V_0-V^*.
$$
\begin{thm}\label{T1.3}
Suppose that
 $(d_n,d_J,d_\mathcal{E})\in H^3(0,1)\times H^2(0,1)\times H^2(0,1)$ and
the doping profile $\rho(x)$ satisfies Assumption  $\mathbf{(A1)}$.
Then the IBVP (\ref{1.2}) -- (\ref{1.3}) has a local solution $(n,
J, \mathcal{E}, V)$ for $t \in (0, t^*)$ satisfying
\begin{eqnarray*}
  n &\in& H^1(0,t^*; H^2(0,1)) \cap L^2(0,t^*; H^4(0,1)), \\
   V &\in& H^1(0,t^*; H^4(0,1)) \cap L^2(0,t^*; H^6(0,1)),  \\
   J &\in& H^1(0,t^*; H^1(0,1)) \cap L^2(0,t^*; H^3(0,1)), \\
   \mathcal{E} &\in& H^1(0,t^*; H^1(0,1)) \cap L^2(0,t^*; H^3(0,1)).
\end{eqnarray*}
\end{thm}

\begin{thm}\label{T1.2}
Let  Assumptions  $\mathbf{(A1)}-\mathbf{(A2)}$ hold, and let
$(n^*(x), J^*(x), \mathcal{E}^*(x), V^*(x))$ be the unique
steady-state solution of the BVP (\ref{1.2}) -- (\ref{1.4}).  In
addition, assume that \begin{eqnarray}\label{newcond1}\displaystyle
\frac{\tau(2+\tau\nu)(2\mu+5)(\mu\nu\lambda^2+\nu\rho(x))}{3\lambda^2}>1,
\quad x \in [0,1].\end{eqnarray} Then the local-in-time solution
$(n(x, t), J(x, t), \mathcal{E}(x, t), V(x, t))$ of the IBVP
(\ref{1.2})--(\ref{1.4}) obtained in Theorem \ref{T1.3} actually
exists globally in time and satisfies
\begin{eqnarray}\label{1.9}
&&\|n(t)-n^*\|_3^2 + \|J(t)- J^*\|_2^2
 + \|\mathcal{E}(t)-\mathcal{E}^*\|_2^2
 + \lambda^2 \|V(t)-V^*\|_5^2 \nonumber \\[1.2ex]
\leqslant &&
Ce^{- \varsigma_1 t}
    \big( \|d_n\|_3^2 + \|d_J\|_2^2 + \|d_\mathcal{E}\|_2^2
                 + \|d_V\|_3^2 \big),
\end{eqnarray}
for some positive constants $\varsigma_1$ and
$C$ provided
there exists $\varsigma_0 > 0$ such that  $$\|d_n\|_3^2 + \|d_J\|_2^2 +
\|d_\mathcal{E}\|_2^2 + \|d_V\|_3^2 \leqslant \varsigma_0.$$
\end{thm}

The   paper is organized as follows.   In Section \ref{sec:ss}, we
prove the existence and uniqueness of the steady-state solution for
the system (\ref{1.2}) with boundary condition (\ref{1.4}). In
Section \ref{sec:ls} we establish the existence of a local-in-time
solution.  In Section \ref{sec:stab}  the global-in-time stability
of the steady-state solution obtained in Section \ref{sec:ss} is
investigated.

\section{Existence and uniqueness of the steady-state solution}\label{sec:ss}
This section is devoted to  proving  Theorem \ref{T1.1}, i.e., we
prove the existence and uniqueness of a steady-state solution
$(n^*,J^*,\mathcal{E}^*,V^*)$ of the BVP  (\ref{1.2}) -- (\ref{1.4})
that satisfies
\begin{eqnarray}\label{1.6}
 \hspace{-0.8in}  \left \{
 \begin{array}{ll}
   - \nu n^*_{xx} + J^*_x = 0,   \medskip\\
  \displaystyle
   - \nu J^*_{xx} + \frac{1}{\tau} J^*
   + \frac23 \left(\frac{(J^*)^2}{n^*}\right)_x
   + \frac23 \mathcal{E}^*_x
   - \frac{\epsilon^2}{9}n^* \left(\frac{(\sqrt{n^*})_{xx}}{\sqrt{n^*}}\right)_x
   - n^*V^*_x + \mu n^*_x
   = 0,  \medskip\\
  \displaystyle
   - \nu \mathcal{E}^*_{xx} + \frac{2}{\tau} \mathcal{E}^*
   - \frac{3}{\tau}n^* + \mu J^*_x
   + \frac53 \left(\frac{J^*}{n^*} \mathcal{E}^* \right)_x
   - \frac13
\left(\frac{(J^*)^3}{(n^*)^2}\right)_x  \medskip \\
  \displaystyle \qquad\qquad\qquad\qquad\qquad\quad\
   - \frac{\epsilon^2}{18}\left(\frac{J^* n^*_{xx}}{n^*} - \frac{J^* (n^*)_x^2}{(n^*)^2}\right)_x
   - J^* V^*_x
   =0,  \medskip \\
  \displaystyle
   -\lambda^2 V^*_{xx} = -n^* + \rho(x).
  \end{array}
 \right.
\end{eqnarray}

The idea is to construct the expected solution
$(n^*,J^*,\mathcal{E}^*)$ around the state $\displaystyle
\left(\rho(x), J_b, \frac32\rho(x)\right)$.  To this end,  set
\begin{eqnarray}\label{2.1}
p(x) = n^*(x) - \rho(x), \ \ q(x) = J^*(x) - J_b, \ \ r(x) =
\mathcal{E}^*(x) - \frac32 \rho(x).
\end{eqnarray}
Then, the BVP (\ref{1.6}) and (\ref{1.4}) can be reformulated as
\begin{eqnarray}\label{2.2}
 \hspace{-0.8in}\left \{
  \begin{array}{ll}
   - \nu p_{xx} + q_x = 0,   \medskip\\
  \displaystyle
   - \left(\nu+\frac{\epsilon^2}{18\nu}\right)q_{xx} + b_2q + c_2p + d_2p_x
   + e_2q_x + \frac23 r_x - \rho(x)V^*_{x}
   = a_2 + f_2,   \medskip \\
     \displaystyle
   - \nu r_{xx} + b_3r + c_3p + d_3p_x + e_3q_x + h_3q
   + \frac53 \frac{J_b}{\rho}r_x
   - J_b V^*_{x}  \medskip \\
  \displaystyle \qquad\qquad\qquad\qquad\qquad\qquad\quad\
   = a_3 + f_3
       + \frac{\epsilon^2}{18\nu} \frac{q_{xx}(q+J_b)}{p+\rho(x)},
          \medskip \\
   - \lambda^2 V^*_{xx} + p = 0, \medskip \\
  \partial_x^k p(0) = \partial_x^k p(1) = 0, \ k=0,1,2,3, \medskip \\
   q(0) = q(1) = r(0) = r(1) = V^*_{x}(0) = V^*_{x}(1) = 0, \ V^*(0) = 0, \
   V^*(1)=V_b,
  \end{array}
 \right. \qquad
\end{eqnarray}
where the function parameters $a_2$, $b_2$, $c_2$, $d_2$, $e_2$,
$f_2$, $a_3$, $b_3$, $c_3$, $d_3$, $e_3$ and $f_3$ are listed in
\ref{app:para} for simplicity of exposition.

The proof  of  Theorem \ref{T1.1} will be done in two steps.
\subsection{A-priori estimate}

The first step is to show that the system
(\ref{2.2}) possesses a solution $(p,q,r,V^*)$    belonging to the following space
\begin{eqnarray}\label{2.3}
\hspace{-0.5in} \mathfrak{X} := \big\{ (p,q,r,V^*)\in H^3\times H^2\times H^2\times
H^5\,  : \
        \|p\|_3^2+\|q\|_2^2+\|r\|_2^2 \leqslant C_0\delta_0^2\big\},
\end{eqnarray}
where $C_0$ is a positive constant to be specified later.  To this
end, we construct a priori estimate for any  $(p,q,r,V^*)$
satisfying the system (\ref{2.2}).

 First, it follows from $(\ref{2.2})_1$ and boundary
conditions that
\begin{eqnarray}\label{2.4}
q = \nu p_x,
\end{eqnarray}
which together with $(\ref{2.2})_4$ yield
\begin{eqnarray}\label{2.5}
 - \int_0^1 \rho(x)V^*_{x} q {\rm d}x
 = \frac{\nu}{\lambda^2} \int_0^1 \rho(x) p^2 {\rm d}x
   + \nu \int_0^1 \rho_x V^*_{x} p {\rm d}x.
\end{eqnarray}
Then, multiplying $(\ref{2.2})_1$ by $d_2 p$ and $(\ref{2.2})_2$ by
$q$, and integrating the resultant equalities over $[0,1]$, we
arrive at
$$
\nu \int_0^1 d_2p_x^2 {\rm d}x
 + \nu \int_0^1 d_{2x}pp_x {\rm d}x
 - \int_0^1 d_2 qp_x {\rm d}x
 - \int_0^1 d_{2x}pq {\rm d}x
 = 0,
$$
and
\begin{eqnarray*}
&& \hspace{-0.9in} (\nu+\frac{\epsilon^2}{18\nu}) \int_0^1 q_x^2 {\rm d}x
 + \int_0^1 b_2q^2 {\rm d}x
 + \int_0^1 c_2pq {\rm d}x
 + \int_0^1 d_2p_xq {\rm d}x
 + \int_0^1 e_2qq_x {\rm d}x   \\
&& \qquad
 + \frac23 \int_0^1 qr_x {\rm d}x
 - \int_0^1 \rho(x) V^*_x q {\rm d}x  \\
 &&\hspace{-1.1in} = (\nu+\frac{\epsilon^2}{18\nu}) \int_0^1 q_x^2 {\rm d}x
 + \int_0^1 b_2q^2 {\rm d}x
 + \int_0^1 c_2pq {\rm d}x
 + \int_0^1 d_2p_xq {\rm d}x
 - \frac12 \int_0^1 e_{2x}q^2 {\rm d}x  \\
&& \qquad
 + \frac23 \int_0^1 qr_x {\rm d}x
 + \frac{\nu}{\lambda^2} \int_0^1 \rho(x) p^2 {\rm d}x
 + \nu \int_0^1 \rho_x V^*_x p {\rm d}x  \\
 &&\hspace{-1.1in} = \int_0^1 (a_2 + f_2) q {\rm d}x,
\end{eqnarray*}
where we have used (\ref{2.5}).  Adding the above two equalities, we
obtain
 \begin{eqnarray}\label{2.6}
&& \hspace{-0.9in}\nu\mu \int_0^1 p_x^2 {\rm d}x
  + (\nu+\frac{\epsilon^2}{18\nu}) \int_0^1 q_x^2 {\rm d}x
  + \frac{1}{\tau} \int_0^1 q^2 {\rm d}x
  + \frac{\nu}{\lambda^2} \int_0^1 \rho(x)p^2 {\rm d}x
  + \frac23 \int_0^1 q r_x {\rm d}x  \nonumber \\
  && \hspace{-1.1in} =- \nu \int_0^1 d_{2x}p p_x {\rm d}x
    + \int_0^1 (d_{2x}-c_2)pq {\rm d}x
    + \frac12 \int_0^1 e_{2x} q^2 {\rm d}x
    - \nu \int_0^1 \rho_x V^*_x p {\rm d}x  \nonumber \\
&&   + \int_0^1 a_2 q {\rm d}x
    + \int_0^1 f_2 q {\rm d}x
    + \int_0^1 (\frac{1}{\tau}-b_2)q^2 {\rm d}x
    + \nu \int_0^1 (\mu-d_2)p_x^2 {\rm d}x.
\end{eqnarray}
Multiplying $(\ref{2.2})_4$ by $\nu\rho V^*$ and integrating the
resultant equality, we get
\begin{eqnarray}\label{2.7}
 \nu\lambda^2 \int_0^1 \rho(x) (V^*_x)^2 {\rm d}x
  + \nu\lambda^2 \int_0^1 \rho_x V^* V^*_x {\rm d}x
  + \nu \int_0^1 \rho(x) p V^* {\rm d}x
  = 0.
\end{eqnarray}

Let $\beta$ be a positive constant that satisfies
$$\displaystyle \left( \frac{(\sqrt{2}+1)\tau}{12}(\mu+\frac52)
- \frac{\nu}{2} \right)\beta <
\left(\nu+\frac{\epsilon^2}{18\nu}\right)(1-\alpha_0)$$ where
$\alpha_0$ is as specified in the Assumption $\mathbf{(A2)}$.  Then
testing $(\ref{2.2})_3$ by $\beta r$ and integrating the resultant
equality to obtain
\begin{eqnarray}\label{2.8}
&&  \hspace{-0.9in}\beta \nu \int_0^1 r_x^2 {\rm d}x
  + \frac{2}{\tau} \beta \int_0^1 r^2 {\rm d}x
  - \frac{3}{\tau}\beta \int_0^1 pr {\rm d}x
  + \beta \int_0^1 (\mu+\frac52)q_x r {\rm d}x \nonumber \\
   &&  \hspace{-1.1in} =- \beta \int_0^1 h_3 qr {\rm d}x
    - \frac56\beta
      \int_0^1 \frac{J_b\rho_x}{\rho^2(x)}r^2 {\rm d}x
    + \beta \int_0^1 J_b V^*_x r {\rm d}x
    - \beta \int_0^1 d_3 p_x r {\rm d}x
    + \beta \int_0^1 a_3 r {\rm d}x  \nonumber \\
&&\quad   + \beta \int_0^1 f_3 r {\rm d}x
    + \frac{\epsilon^2}{18\nu}\beta
      \int_0^1 \frac{q_{xx}(q+J_b)}{p+\rho(x)}r {\rm d}x
    + \beta \int_0^1 (\frac{2}{\tau}-b_3)r^2 {\rm d}x  \nonumber \\
&&\quad  - \beta \int_0^1 (\frac{3}{\tau}+c_3)pr {\rm d}x
    - \beta \int_0^1 (e_3-\mu-\frac52)q_x r {\rm d}x.
\end{eqnarray}
Now combining (\ref{2.6})--(\ref{2.8}),  one can show that for  any
arbitrary positive constant $C$,
\begin{eqnarray}\label{2.9}
&&  \hspace{-0.9in} \frac{\nu}{\lambda^2}C \int_0^1 \rho(x)p^2 {\rm d}x
  + \frac{C}{\tau} \int_0^1 q^2 {\rm d}x
  + \frac{2\beta}{\tau} \int_0^1 r^2 {\rm d}x
  + \nu\mu C \int_0^1 p_x^2 {\rm d}x
  + (\nu+\frac{\epsilon^2}{18\nu})C \int_0^1 q_x^2 {\rm d}x
    \nonumber \\
&&\hspace{-0.8in}+ \beta\nu \int_0^1 r_x^2 {\rm d}x
  + \nu\lambda^2 C \int_0^1 \rho(x)(V^*_x)^2 {\rm d}x
  + \frac23 C \int_0^1 qr_x {\rm d}x
  - \frac{3}{\tau}\beta \int_0^1 pr {\rm d}x  \nonumber \\
&& \hspace{-0.8in}+ \beta \int_0^1 (\mu+\frac52)q_x r {\rm d}x
  + \nu C \int_0^1 \rho(x)p V^* {\rm d}x  \nonumber \\
  && \hspace{-1.1in} = - \nu C \int_0^1 d_{2x}p p_x {\rm d}x
    + C \int_0^1 (d_{2x}-c_2)pq {\rm d}x
    + \frac{C}{2} \int_0^1 e_{2x} q^2 {\rm d}x
    - \nu C \int_0^1 \rho(x)V^*_x p {\rm d}x \nonumber \\
&& \hspace{-0.8in}  + C \int_0^1 a_2 q {\rm d}x
    + C \int_0^1 f_2 q {\rm d}x
    + C \int_0^1 (\frac{1}{\tau}-b_2)q^2 {\rm d}x
    + \nu C \int_0^1 (\mu-d_2)p_x^2 {\rm d}x  \nonumber \\
&& \hspace{-0.8in}  - \nu \lambda^2 C \int_0^1 \rho_x V^* V^*_x {\rm d}x
    - \beta \int_0^1 h_3 qr {\rm d}x
    - \frac56\beta
      \int_0^1 \frac{J_b\rho_x}{\rho^2(x)}r^2 {\rm d}x
    + \beta \int_0^1 J_b V^*_x r {\rm d}x  \nonumber \\
&&   \hspace{-0.8in} - \beta \int_0^1 d_3 p_x r {\rm d}x
    + \beta \int_0^1 a_3 r {\rm d}x
    + \beta \int_0^1 f_3 r {\rm d}x
    + \frac{\epsilon^2}{18\nu}\beta
      \int_0^1 \frac{q_{xx}(q+J_b)}{p+\rho(x)}r {\rm d}x \nonumber \\
&&   \hspace{-0.8in} + \beta \int_0^1 (\frac{2}{\tau}-b_3)r^2 {\rm d}x
    - \beta \int_0^1 (\frac{3}{\tau}+c_3)pr {\rm d}x
    - \beta \int_0^1 (e_3-\mu-\frac52)q_x r {\rm d}x.
\end{eqnarray}
  Observe that
\begin{eqnarray*}
\hspace{-0.8in} \|V^*\|^2 = \int_0^1 \big| \int_0^x V^*_{y} {\rm d}y
\big|^2 {\rm d}x
  \leqslant \|V^*_x\|^2, \quad
  \left( (\nu \rho(x))^2
     - 4\cdot\frac{\nu\rho(x)}{2\lambda^2}
        \cdot\frac{2}{3}\nu\lambda^2\rho(x) \right)C^2 <0 .
   \end{eqnarray*}
    Thus
\begin{eqnarray}\label{2.10}
&& \frac{\nu}{2\lambda^2}C \int_0^1 \rho(x)p^2 {\rm d}x
  + \frac{2\nu\lambda^2}{3} C \int_0^1 \rho(x) (V^*_x)^2 {\rm d}x
  + \nu C \int_0^1 \rho(x)p V^* {\rm d}x \nonumber\\
 & \geqslant &
   \frac{\nu}{2\lambda^2}C \int_0^1 \rho(x)p^2 {\rm d}x
    + \frac{2\nu\lambda^2}{3} C \int_0^1 \rho(x)(V^*)^2 {\rm d}x
    + \nu C \int_0^1 \rho(x) p V^* {\rm d}x \nonumber\\
  &  > & 0.
\end{eqnarray}
In particular, under the assumption $\mathbf{(A2)}$, by taking
$\displaystyle C = \frac{3\sqrt{2}}{2}(\mu+\frac52)\beta$ in
(\ref{2.9}), it is straightforward to show that there exist positive
constants $l_i, i= 1,2,3$, such that
\begin{eqnarray}\label{2.11}
&&\hspace{-0.9in} l_1 \|p\|_1^2 + l_2 \|q\|_1^2 + l_3 \|r\|_1^2 \nonumber \\
  && \hspace{-1.1in}\leqslant
   \frac{C\nu}{2\lambda^2} \int_0^1 \rho(x)p^2 {\rm d}x
    + \frac{C}{\tau} \int_0^1 q^2 {\rm d}x
    + \frac{2\beta}{\tau} \int_0^1 r^2 {\rm d}x
    + \nu\mu C \int_0^1 p_x^2 {\rm d}x
    + (\nu+\frac{\epsilon^2}{18\nu})C \int_0^1 q_x^2 {\rm d}x
    \nonumber \\
&& \hspace{-0.9in}  + \beta \nu \int_0^1 r_x^2 {\rm d}x
    + \frac23 C \int_0^1 qr_x {\rm d}x
    - \frac{3}{\tau}\beta \int_0^1 pr {\rm d}x
    + \beta \int_0^1 (\mu+\frac52)q_xr {\rm d}x.
\end{eqnarray}
Inserting (\ref{2.10}) and (\ref{2.11}) into (\ref{2.9}), this
yields
\begin{eqnarray*}
&&\hspace{-0.9in} l_1\|p\|_1^2 + l_2\|q\|_1^2 + l_3\|r_1\|_1^2
  + \frac{\nu\lambda^2}{3}C \int_0^1 \rho(x)(V^*_x)^2 {\rm d}x \\
&& \hspace{-1.1in} \leqslant
  - \nu C \int_0^1 d_{2x}pp_x {\rm d}x
  + C \int_0^1 (d_{2x}-c_2)pq {\rm d}x
    + \frac{C}{2} \int_0^1 e_{2x} q^2 {\rm d}x
    - \nu C \int_0^1 \rho(x)V^*_x p {\rm d}x  \\
&&\hspace{-0.9in}   + C \int_0^1 a_2 q {\rm d}x
    + C \int_0^1 f_2 q {\rm d}x
    + C \int_0^1 (\frac{1}{\tau}-b_2)q^2 {\rm d}x
    + \nu C \int_0^1 (\mu-d_2)p_x^2 {\rm d}x   \\
&&\hspace{-0.9in}   - \nu \lambda^2 C \int_0^1 \rho_x V^* V^*_x {\rm d}x
    - \beta \int_0^1 h_3 qr {\rm d}x
    - \frac56\beta
      \int_0^1 \frac{J_b\rho_x}{\rho^2(x)}r^2 {\rm d}x
    + \beta \int_0^1 J_b V^*_x r {\rm d}x   \\
&& \hspace{-0.9in}  - \beta \int_0^1 d_3 p_x r {\rm d}x
    + \beta \int_0^1 a_3 r {\rm d}x
    + \beta \int_0^1 f_3 r {\rm d}x
    + \frac{\epsilon^2}{18\nu}\beta
      \int_0^1 \frac{q_{xx}(q+J_b)}{p+\rho(x)}r {\rm d}x  \\
&&  \hspace{-0.9in} + \beta \int_0^1 (\frac{2}{\tau}-b_3)r^2 {\rm d}x
    - \beta \int_0^1 (\frac{3}{\tau}+c_3)pr {\rm d}x
    - \beta \int_0^1 (e_3-\mu-\frac52)q_x r {\rm d}x   \\
 &&  \hspace{-1.1in} \leqslant
  C_1\delta_0 (\|p\|_1^2+\|q\|_1^2+\|r\|_1^2+\|V^*\|_1^2)
   + \hat{l}_1\|q\|^2 + C_{\hat{l}_1}\|a_2\|^2
   + \hat{l}_2\|r\|^2 + C_{\hat{l}_2}\|a_3\|^2   \\
&& \hspace{-0.9in} + \frac23 \nu C J_b^2 \int_0^1 \rho^{-2}(x)p_x^2 {\rm d}x
   - 2\beta J_b^2 \int_0^1 \rho^{-3}(x) qr {\rm d}x
   + \beta J_b \int_0^1 V^*_x r {\rm d}x   \\
&& \hspace{-0.9in} + \frac{5\beta J_b}{2} \int_0^1 \rho^{-1}(x)p_xr {\rm d}x
   - \frac{\epsilon^2}{18\nu}\beta J_b
     \int_0^1 \frac{q_x r}{p+\rho(x)} {\rm d}x
   + \beta J_b^2 \int_0^1 \rho^{-2}(x)q_x r {\rm d}x,
\end{eqnarray*}
where $\displaystyle \hat{l}_1=\frac{l_2}{2},
\hat{l}_2=\frac{l_3}{2}, C_{\hat{l}_1}$ and $C_{\hat{l}_2}$ are
positive constants. Therefore, provided that $J_b$ and $\delta_0$
are sufficiently small, there exist positive constants $\tilde{l}_i,
i=1,2,3,4$ such that
\begin{eqnarray}\label{2.12}
 \tilde{l}_1 \|p\|_1^2 + \tilde{l}_2\|q\|_1^2
  + \tilde{l}_3 \|r\|_1^2 + \tilde{l}_4 \|V^*\|_1^2
 \leqslant
  C_2 \delta_0^2.
\end{eqnarray}

Following above similar analysis, we can obtain the high order
estimate
\begin{eqnarray}\label{2.13}
 \|p\|_2^2 + \|q\|_2^2 + \|r\|_2^2 + \|V^*\|_2^2
 \leqslant
  C_3\delta_0^2.
\end{eqnarray}
Moreover, thanks to $(\ref{2.2})_4$, (\ref{2.4}) and (\ref{2.13}),
we conclude that there exists a positive constant $C_0$ such that
\begin{eqnarray}\label{2.14}
 \|p\|_3^2 + \|q\|_2^2 + \|r\|_2^2 + \|V^*\|_5^2
 \leqslant
  C_0\delta_0^2.
\end{eqnarray}

\subsection{Existence and uniqueness of  solutions to (\ref{2.2})}

In this subsection we   establish  the existence of solutions for
BVP (\ref{2.2}).   For simplicity of notations, define the left-hand
sides of (\ref{2.2}) as $\mathfrak{L}(w)$ where $w=(p,q,r,V^*)$.
Clearly  $\mathfrak{L}$ is a linear operator from $\mathcal{H}:=
\big\{ w\ \big| \ w\in H^2\times H^2 \times H^2 \times H^2 \big\}$
to $\mathcal{L}:= \big\{ u \ \big| \ u \in L^2 \times L^2 \times L^2
\times L^2 \big\}$. It can be shown that for any
$F=(F_1,F_2,F_3,F_4)\in \mathcal{L}$, there exists  a corresponding
unique solution $w=(p,q,r,V^*)\in\mathcal{H}$ to the equation
$\mathfrak{L}(w)=F$.

Define another linear operator $\mathfrak{L}_{\theta}$ by
\begin{eqnarray}\label{2.15}
 \mathfrak{L}_{\theta}(w)
  = \theta \mathfrak{L}(w) + (1-\theta)\mathfrak{L}_0(w),
\end{eqnarray}
where $\mathfrak{L}_0(w)= -w_{xx} + w$ and $\theta \in I := [0,1]$.   Let
$\tilde{I}$ be the collection of all  $\theta \in I$ such that for any
$F(x)\in\mathcal{L}$ there exists a unique solution $w \in
\mathcal{H}$  for the equation $\mathfrak{L}_{\theta}(w) = F(x)$.
 Then $\tilde{I}$ is nonempty, because $0\in\tilde{I}$.
We next prove  that $1\in\tilde{I}$.

Suppose that $\theta_0 \in \tilde{I}$.   Then for any
$u\in\mathcal{H}, F\in \mathcal{L}$, there is a unique corresponding
function $w\in\mathcal{H}$ such that
\begin{eqnarray}\label{2.17}
 \mathfrak{L}_{\theta_0}(w)
  = F + (\theta_0-\theta)\mathfrak{L}(u) - (\theta_0-\theta)\mathfrak{L}_0(u).
\end{eqnarray}
Define the mapping  $\T_{\theta_0, \theta}: \mathcal{H} \to \mathcal{H}$ by
\begin{eqnarray*}
 \T_{\theta_0, \theta}(u) = w, \qquad u\in\mathcal{H}.
\end{eqnarray*}
Then for any  $u_1, u_2 \in \mathcal{H}$, there exist
two corresponding functions $w_1, w_2 \in \mathcal{H}$ satisfying
$$
 \mathfrak{L}_{\theta_0}(w_i)
  = F + (\theta_0-\theta)\mathfrak{L}(u_i) - (\theta_0-\theta)\mathfrak{L}_0(u_i) \
  \mbox{and} \ \T_{\theta_0, \theta}(u_i) = w_i, \ i=1,2.
$$
Since the operators $\mathfrak{L}_{\theta_0}, \mathfrak{L}_0$ and
$\mathfrak{L}$ are linear, we have
$$
 \mathfrak{L}_{\theta_0}(w_2-w_1)
  = (\theta_0-\theta) \big(\mathfrak{L}(u_2-u_1) - \mathfrak{L}_0(u_2-u_1)\big).
$$
Moreover, following the same proof process to reach (\ref{2.14}), we
can deduce that there exists a positive constant $C_M$ such that
\begin{eqnarray*}
 \|w_2-w_1\|_{\mathcal{H}}^2
  \leqslant
   C_M |\theta_0-\theta| \|u_2-u_1\|_{\mathcal{H}}^2,
\end{eqnarray*}
which implies that the operator $\T_{\theta_0, \theta}$ is a
contractive map on $\mathcal{H}$ if $C_M |\theta_0-\theta|<1$.

Without loss of generality, suppose that $w^*$ is a fixed point for
$\T_{\theta_0, \theta}$.   It then follows from (\ref{2.15}) and
(\ref{2.17}) that
 $\mathfrak{L}_{\theta}(w^*) = F$ and moreover  $\theta \in \tilde{I}$.
 Therefore,   $\tilde{I}$ can be extended to the whole interval $[0,1]$
by the method of continuity.  In particular, $1\in\tilde{I}$.  Therefore
  the  equation
$\mathfrak{L}(w)=F$   has a  unique solution $w\in\mathcal{H}$  for any $F\in \mathcal{L}$.

We are now ready to show the existence of solutions to the BVP
(\ref{2.2}).  In fact, due to  the arguments above,   for any
$$u=(u_1,u_2,u_3,u_4)\in\mathcal{H},\quad \displaystyle F= \left(0, a_2+f_2,
a_3+f_3
 + \frac{\epsilon^2}{18\nu}\frac{u_{2xx}(u_2+J_b)}{u_1+\rho(x)},0 \right),$$
there exists a unique solution $w=(p,q,r,V^*)\in\mathcal{H}$ for the
equation $\mathfrak{L}(w) = F$.

Let $\tilde{\T}(u)=w$, we claim that $\tilde{\T}$ is a contractive
map. Indeed, for any $u^{(i)}=(u_1^{(i)}, u_2^{(i)}, u_3^{(i)},
u_4^{(i)})$, there corresponds to a function $w_i$ such that
$\tilde{\T}(u^{(i)}) = w_i, i=1,2$ and
$$
 \mathfrak{L}(w_i) = \left(0, a_2+f_2(u^{(i)}),
           a_3+f_3(u^{(i)})+\frac{\epsilon^2}{18\nu}
           \frac{u_{2xx}^{(i)}(u_2^{(i)}+J_b)}{u_1^{(i)}+\rho(x)}, 0 \right),
           \ i=1,2.
$$
Thus
$$\mathfrak{L}(w_1-w_2)
 =  \left(0, f_2(u^{(1)})-f_2(u^{(2)}),
     f_3(u^{(1)})+\frac{\epsilon^2}{18\nu} \tilde{f}_3 (u^{(1)},u^{(2)})
     -f_3(u^{(2)}),
     0\right),
$$
where $\tilde{f}_3$ is given in \ref{app:para}.

From the explicit formula of $f_i, i=2,3$, it is straightforward to
check that
\begin{eqnarray*}
 |f_i(u^{(1)})-f_i(u^{(2)})|
  \leqslant
   C_4 \delta_0 (|u^{(1)}-u^{(2)}| + |u_x^{(1)}-u_x^{(2)}|).
\end{eqnarray*}
Furthermore, we can obtain
\begin{eqnarray*}
 \|w_1-w_2\|_{\mathcal{H}}^2
  \leqslant
   C_5 \delta_0 \|u^{(1)}-u^{(2)}\|_{\mathcal{H}}^2.
\end{eqnarray*}
Consequently, the operator $\tilde{\T}$ is a contractive map from
$\mathcal{H} \mapsto \mathcal{H}$ as long as $C_5\delta_0 < 1$,
which means there exists a unique fixed point
$w=(p,q,r,V^*)\in\mathcal{H}$ satisfying
\begin{eqnarray*}
\mathfrak{L}(w)
  = \left(0, a_2+f_2(w),
     a_3+f_3(w)+\frac{\epsilon^2}{18\nu}\frac{q_{xx}(q+J_b)}{p+\rho(x)},
     0 \right).
\end{eqnarray*}
Moreover, it follows from $(\ref{2.2})_1$ and $(\ref{2.2})_4$ that
$w\in\mathfrak{X}$, which is the unique solution of BVP (\ref{2.2}).
Then, we can conclude from (\ref{2.1}) that there exists a unique
solution $(n^*(x), J^*(x), \mathcal{E}^*(x), V^*(x))$ to the BVP
(\ref{1.6}) and (\ref{1.4}), provided $\delta_0$ and $J_b$ are
sufficiently small, that is
\begin{eqnarray*}
 n^*(x) = p(x)+\rho(x), \ J^*(x) = q(x)+J_b, \
 \mathcal{E}^*(x) = r(x)+\frac32\rho(x),
\end{eqnarray*}
which satisfies (\ref{1.8}). Hence the proof of Theorem \ref{T1.1}
is complete.

\section{Local existence of solutions} \label{sec:ls}

The purpose of this section is to prove Theorem \ref{T1.3}.   To that end, introduce the change of variables
\begin{equation}\label{3.1}
\tn=n-n^*, \ \tJ=J-J^*, \ \tE=\mathcal{E}-\mathcal{E}^*, \ \tV=V-V^*,
\end{equation}
where $(n^*(x), J^*(x), \mathcal{E}^*(x), V^*(x))$ is the
steady-state solution of BVP (\ref{1.6}) and (\ref{1.4}) shown to exist in
Section \ref{sec:ss}.
Then the evolution equations for $(\tn, \tJ, \tE, \tV)$ read
\begin{eqnarray}\label{3.2}
\left\{
 \begin{array}{ll}
 \tn_t - \nu\tn_{xx} + \tJ_x = 0, \medskip \\
 \displaystyle
 \tJ_t -\nu \tJ_{xx} + \frac{1}{\tau}\tJ
  - \frac{\epsilon^2}{18}\tn_{xxx}
  + \frac23 \tE_x
  - n_0 \tV_x
  + \mu\tn_x
 = g_1, \medskip \\
 \displaystyle
 \tE_t - \nu \tE_{xx}
  + \frac{2}{\tau}\tE
  - \frac{3}{\tau}\tn
  + (\mu + \frac52) \tJ_x
 = g_2, \medskip \\
 - \lambda^2 \tV_{xx} = - \tn,
 \end{array}
\right.
\end{eqnarray}
where $g_1$ and $g_2$ are listed in \ref{app:para},  along with the
following natural initial and boundary data
\begin{equation}\label{3.3}
 \hspace{-0.5in}\left\{
  \begin{array}{ll}
   \tn(x,0)=\tn_0(x), \ \tJ(x,0)=\tJ_0(x), \ \tE(x,0)=\tE_0(x),\
    \tV(x,0)=\tV_0(x), \medskip \\
   \tV(0,t)=\tV(1,t)=\partial_x \tV(0,t)=\partial_x \tV(1,t) =0,  \medskip \\
   \partial_x^k \tn(0,t) = \partial_x^k \tn(1,t) = 0, \ k=0,1,2,3, \medskip \\
   \partial_x^l \tJ(0,t) = \partial_x^l \tJ(1,t)
   = \partial_x^l \tE(0,t) = \partial_x^l \tE(1,t)=0, \ l=0,1,2, \ t \in [0,1].
  \end{array}
 \right.
\end{equation}

It suffices to establish the local existence of solutions of IBVP
(\ref{3.2})-(\ref{3.3}).   To this end, choose a sequence of numbers
$\{\sigma\}$ with $0<\sigma<1$ and consider the following family of
parabolic initial and boundary value problems for $(x,t)\in
[0,1]\times \mathbb{R}$:
\begin{eqnarray}\label{4.1}
\hspace{-0.5in} \left\{
  \begin{array}{ll}
   \partial_t \tn_{\sigma} + \tJ_{\sigma x}
   = \nu \tn_{\sigma xx} - \sigma\tn_{\sigma xxxx}, \medskip \\
   \displaystyle \partial_t \tJ_{\sigma} - \frac{\epsilon^2}{18} \tn_{\sigma xxx}
     + \frac23 \tE_{\sigma x} - {n^*}\tV_{\sigma x} + \mu\tn_{\sigma x}
     - g_{\sigma 1}
   = \nu \tJ_{\sigma xx} - \sigma \tJ_{\sigma xxxx}
     - \frac{1}{\tau} \tJ_{\sigma}, \medskip \\
   \displaystyle \partial_t \tE_{\sigma} - \frac{3}{\tau}\tn_{\sigma}
     + (\mu+\frac52) \tJ_{\sigma x} - g_{\sigma 2}
   = \nu \tE_{\sigma xx} - \sigma \tE_{\sigma xxxx}
     - \frac{2}{\tau} \tE_{\sigma}, \medskip \\
   - \lambda^2 \tV_{\sigma xx} = - \tn_{\sigma}, \medskip \\
   \tn_{\sigma}(x,0) = \tn_{\sigma}^0, \
   \tJ_{\sigma}(x,0) = \tJ_{\sigma}^0, \
   \tE_{\sigma}(x,0) = \tE_{\sigma}^0,  \medskip \\
   \partial_x^k \tn_{\sigma}(0,t)
    = \partial_x^k\tn_{\sigma}(1,t) = 0, \ k=0,1,2,3,  \medskip \\
   \partial_x^l \tJ_{\sigma}(0,t) = \partial_x^l \tJ_{\sigma}(1,t)
    = \partial_x^l \tE_{\sigma}(0,t) = \partial_x^l \tE_{\sigma}(1,t), \
    l=0,1,2,  \medskip \\
  \tV_{\sigma}(0,t) = \tV_{\sigma}(1,t)
   =\partial_x \tV(0,t)=\partial_x \tV(1,t) = 0,
  \end{array}
 \right.
\end{eqnarray}
where $g_{\sigma 1}$ and $g_{\sigma 2}$ are listed in \ref{app:para}, and
  $(\tn_{\sigma}^0, \tJ_{\sigma}^0, \tE_{\sigma}^0) \in
H^3\times H^2 \times H^2$ is assumed to satisfy compatibility conditions on the
boundary, and moreover
$$
 \tn_{\sigma}^0 \rightarrow \tn_0 \ \ \mbox{in} \ \ H^3(0,1), \quad
 \tJ_{\sigma}^0 \rightarrow \tJ_0 \ \ \mbox{in} \ \ H^2(0,1), \quad
 \tE_{\sigma}^0 \rightarrow \tE_0 \ \ \mbox{in} \ \ H^2(0,1), \quad \mbox{as}\quad \sigma \to 0.
$$

Notice that the system (\ref{4.1}) is a fourth order nonlinear parabolic system
with third order terms.   It is standard to show that this problem has
a unique solution $(\tn_{\sigma}, \tJ_{\sigma}, \tE_{\sigma},
\tV_{\sigma})(x,t)$ in $[0,1]\times[0,t_{\sigma}]$ for some
$t_{\sigma} > 0$ (see \cite{T88,CD07}), which belongs to
\begin{eqnarray}\label{4.2}
 \left\{
 \begin{array}{ll}
 \tn_{\sigma} \in L^2(0,t_{\sigma}; H^5(0,1))
    \cap \C([0,t_{\sigma}]; H^3(0,1)), \medskip \\
 \tJ_{\sigma} \in L^2(0,t_{\sigma}; H^4(0,1))
    \cap \C([0,t_{\sigma}]; H^2(0,1)), \medskip \\
 \tE_{\sigma} \in L^2(0,t_{\sigma}; H^4(0,1))
    \cap \C([0,t_{\sigma}]; H^2(0,1)), \medskip \\
 \tV_{\sigma} \in L^2(0,t_{\sigma}; H^7(0,1))
    \cap \C([0,t_{\sigma}]; H^5(0,1)),
 \end{array}
 \right.
\end{eqnarray}
The proof of Theorem \ref{T1.3} will be done in
four steps.

\subsection{Uniform a-priori estimates.}

To begin with, based on (\ref{4.2}), we can choose a number
$\gamma_0$ such that
\begin{eqnarray*}
 0 < \gamma_0 < \inf\limits_{x\in[0,1]}\tn_0(x),
 \quad\
 \|\tn_0(x)\|_3^2
  + \|\tJ_0(x)\|_2^2
  + \|\tE_0(x)\|_2^2
  < \gamma_0^{-1}
\end{eqnarray*}
and
\begin{eqnarray}\label{4.3}
 \left\{
 \begin{array}{ll}
\gamma_0 \leqslant \inf\limits_{x\in[0,1]}\tn_{\sigma}(x,t),
\medskip \\
  \|\tn_{\sigma}(\cdot, t)\|_3^2
  + \|\tJ_{\sigma}(\cdot, t)\|_2^2
  + \|\tE_{\sigma}(\cdot, t)\|_2^2
  \leqslant \gamma_0^{-1}, \ \ \mbox{for} \  t \in [0,
  t_{\sigma}], \medskip \\
  \max\limits_{t \in [0,t_{\sigma}]}
  \big( \|\tn_{\sigma}(\cdot, t)\|_3^2
  + \|\tJ_{\sigma}(\cdot, t)\|_2^2
  + \|\tE_{\sigma}(\cdot, t)\|_2^2 \big)
  = \gamma_0^{-1}.
 \end{array}
 \right.
\end{eqnarray}
Otherwise, we compress the interval $[0, t_{\sigma}]$.
 In addition, throughout this section it is
assumed that there exist positive constants $\beta_1$ and $\gamma_1$
such that
\begin{eqnarray*}
 \gamma_1 \leqslant {n^*}+\tn_{\sigma}(x,t)
    \leqslant \gamma_1^{-1}, \quad
   \|J^*+\tJ_{\sigma}\|_{L^{\infty}} \leqslant \beta_1.
\end{eqnarray*}
In the following, we concentrate on establishing uniform estimates
of $(\tn_{\sigma},\tJ_{\sigma},\tE_{\sigma},\tV_{\sigma})$ with
respect to $\sigma$, which will imply that the life span
$t_{\sigma}$ can not tend to zero for $\sigma$ going to zero.

 Multiplying $(\ref{4.1})_1,
(\ref{4.1})_2$ and $(\ref{4.1})_3$ by $\mu\tn_{\sigma},
\tJ_{\sigma}$ and $\displaystyle \frac{2}{3\mu} \tE_{\sigma}$,
respectively, then integrating them over [0,1] to obtain
\begin{eqnarray*}
&& \hspace{-0.9in}\frac{\mu}{2} \frac{\rm d}{{\rm d}t} \|\tn_{\sigma}\|^2
  + \mu \int_0^1 \tJ_{\sigma x} \tn_{\sigma} {\rm d}x \\
& &\hspace{-1.1in} =\mu\nu \int_0^1 \tn_{\sigma xx}\tn_{\sigma} {\rm d}x
    - \mu\sigma \int_0^1 \tn_{\sigma xxxx}\tn_{\sigma} {\rm d}x
 = - \mu\nu \|\tn_{\sigma x}\|^2
    - \mu\sigma \|\tn_{\sigma xx}\|^2, \\
&& \hspace{-0.9in}\frac12\frac{\rm d}{{\rm d}t} \|\tJ_{\sigma}\|^2
  - \frac{\epsilon^2}{18} \int_0^1 \tn_{\sigma xxx}\tJ_{\sigma} {\rm d}x
  + \frac23 \int_0^1 \tE_{\sigma x} \tJ_{\sigma} {\rm d}x \\
 && \hspace{-0.7in} - \int_0^1 {n^*}\tV_{\sigma x} \tJ_{\sigma} {\rm d}x
  + \mu \int_0^1 \tn_{\sigma x} \tJ_{\sigma} {\rm d}x
  - \int_0^1 g_{\sigma 1} \tJ_{\sigma} {\rm d}x   \\
  && \hspace{-1.1in}= \nu \int_0^1 \tJ_{\sigma xx} \tJ_{\sigma} {\rm d}x
    - \sigma \int_0^1 \tJ_{\sigma xxxx} \tJ_{\sigma} {\rm d}x
    - \frac{1}{\tau} \int_0^1 \tJ_{\sigma}^2 {\rm d}x \\
   && \hspace{-1.1in}= - \nu \|\tJ_{\sigma x}\|^2
    - \sigma \|\tJ_{\sigma xx}\|^2
    - \frac{1}{\tau}\|\tJ_{\sigma}\|^2,
\end{eqnarray*}
and
\begin{eqnarray*}
&&\hspace{-0.9in} \frac{1}{3\mu} \frac{\rm d}{{\rm d}t} \|\tE_{\sigma}\|^2
  - \frac{2}{\mu\tau} \int_0^1 \tn_{\sigma} \tE_{\sigma} {\rm d}x
  + \frac23 \int_0^1 \tJ_{\sigma x} \tE_{\sigma} {\rm d}x
  - \frac{2}{3\mu} \int_0^1 (g_{\sigma 2} - \frac52 \tJ_{\sigma x})\tE_{\sigma} {\rm d}x
     \\
  && \hspace{-1.1in}=\frac{2\nu}{3\mu} \int_0^1 \tE_{\sigma xx}\tE_{\sigma} {\rm d}x
    - \frac{2\sigma}{3\mu} \int_0^1 \tE_{\sigma xxxx}\tE_{\sigma} {\rm d}x
    - \frac{4}{3\mu\tau} \int_0^1 \tE_{\sigma}^2 {\rm d}x   \\
  && \hspace{-1.1in}= - \frac{2\nu}{3\mu} \|\tE_{\sigma x}\|^2
    - \frac{2\sigma}{3\mu} \|\tE_{\sigma xx}\|^2
    - \frac{4}{3\mu\tau} \|\tE_{\sigma}\|^2.
\end{eqnarray*}
Adding the above three equalities results in
\begin{eqnarray}\label{4.4}
&&\hspace{-0.5in} \frac12 \frac{\rm d}{{\rm d}t}
  \big( \mu \|\tn_{\sigma}\|^2
        + \|\tJ_{\sigma}\|^2
        + \frac{2}{3\mu} \|\tE_{\sigma}\|^2 \big)
  + \mu\nu \|\tn_{\sigma x}\|^2
  + \mu\sigma \|\tn_{\sigma xx}\|^2 \nonumber \\
&&\hspace{-0.4in} + \frac{1}{\tau} \|\tJ_{\sigma}\|^2
  + \nu \|\tJ_{\sigma x}\|^2
  + \sigma \|\tJ_{\sigma xx}\|^2
  + \frac{4}{3\mu\tau} \|\tE_{\sigma}\|^2
  + \frac{2\nu}{3\mu} \|\tE_{\sigma x}\|^2
  + \frac{2\sigma}{3\mu} \|\tE_{\sigma xx}\|^2 \nonumber \\
  &&\hspace{-0.7in}= \frac{\epsilon^2}{18} \int_0^1 \tn_{\sigma xxx}\tJ_{\sigma} {\rm d}x
    + \int_0^1 {n^*} \tV_{\sigma x} \tJ_{\sigma} {\rm d}x
    + \frac{2}{\mu\tau} \int_0^1 \tn_{\sigma} \tE_{\sigma} {\rm d}x
    + \int_0^1 g_{\sigma 1} \tJ_{\sigma} {\rm d}x \nonumber \\
&&\hspace{-0.4in}   + \frac{2}{3\mu}
      \int_0^1 (g_{\sigma 2}-\frac52 \tJ_{\sigma x})\tE_{\sigma} {\rm d}x.
\end{eqnarray}
Using partial integration and  $(\ref{4.1})_{1,4}$, we
find that
\begin{eqnarray*}
&& \frac{\epsilon^2}{18} \int_0^1 \tn_{\sigma xxx} \tJ_{\sigma} {\rm d}x
  =  - \frac{\epsilon^2}{18} \int_0^1 \tn_{\sigma xx}\tJ_{\sigma
x} {\rm d}x \\
  &= &- \frac{\epsilon^2}{18} \int_0^1 \tn_{\sigma xx}
       (\nu\tn_{\sigma xx}-\sigma\tn_{\sigma xxxx}-\tn_{\sigma t})
       {\rm d}x \nonumber \\
  &= & - \frac{\nu\epsilon^2}{18} \|\tn_{\sigma xx}\|^2
    - \frac{\epsilon^2}{18}\sigma \|\tn_{\sigma xxx}\|^2
    - \frac{\epsilon^2}{36} \frac{\rm d}{{\rm d}t}
      \|\tn_{\sigma x}\|^2,
\end{eqnarray*}
which can be inserted into (\ref{4.4}) to obtain
\begin{eqnarray}\label{4.5}
&&\hspace{-0.9in} \frac12 \frac{\rm d}{{\rm d}t}
  \big( \mu\|\tn_{\sigma}\|^2
        + \frac{\epsilon^2}{18}\|\tn_{\sigma x}\|^2
        + \|\tJ_{\sigma}\|^2
        + \frac{2}{3\mu}\|\tE_{\sigma}\|^2 \big)
  + \mu\nu \|\tn_{\sigma x}\|^2
  + (\frac{\nu\epsilon^2}{18}+\mu\sigma) \|\tn_{\sigma xx}\|^2 \nonumber \\
&& \hspace{-0.6in}+ \frac{\epsilon^2}{18}\sigma \|\tn_{\sigma xxx}\|^2
  + \frac{1}{\tau} \|\tJ_{\sigma}\|^2
  + \nu \|\tJ_{\sigma x}\|^2
  + \sigma \|\tJ_{\sigma xx}\|^2 \nonumber \\
&& \hspace{-0.6in}
  + \frac{4}{3\mu\tau} \|\tE_{\sigma}\|^2
  + \frac{2\nu}{3\mu} \|\tE_{\sigma x}\|^2
  + \frac{2\sigma}{3\mu} \|\tE_{\sigma xx}\|^2 \nonumber \\
  &&\hspace{-1.1in} = \int_0^1 {n^*} \tV_{\sigma x}\tJ_{\sigma} {\rm d}x
    + \frac{2}{\mu\tau} \int_0^1 \tn_{\sigma} \tE_{\sigma} {\rm d}x
    + \int_0^1 g_{\sigma 1} \tJ_{\sigma} {\rm d}x
    + \frac{2}{3\mu} \int_0^1 (g_{\sigma 2}-\frac52 \tJ_{\sigma x})\tE_{\sigma} {\rm d}x.
\end{eqnarray}

Following above similar procedure, we can get higher order estimates
\begin{eqnarray}\label{4.6}
&&\hspace{-0.9in} \frac12 \frac{\rm d}{{\rm d}t} \big( \mu\|\tn_{\sigma x}\|^2
        + \frac{\epsilon^2}{18}\|\tn_{\sigma xx}\|^2
        + \|\tJ_{\sigma x}\|^2
        + \frac{2}{3\mu}\|\tE_{\sigma x}\|^2 \big)
  + \mu\nu \|\tn_{\sigma xx}\|^2
 \nonumber \\
&&\hspace{-0.6in}   + (\frac{\nu\epsilon^2}{18}+\mu\sigma) \|\tn_{\sigma xxx}\|^2 + \frac{\epsilon^2}{18}\sigma \|\tn_{\sigma xxxx}\|^2
  + \frac{1}{\tau} \|\tJ_{\sigma x}\|^2
  + \nu \|\tJ_{\sigma xx}\|^2
  + \sigma \|\tJ_{\sigma xxx}\|^2   \nonumber \\
&&\hspace{-0.6in}
  + \frac{4}{3\mu\tau} \|\tE_{\sigma x}\|^2
  + \frac{2\nu}{3\mu} \|\tE_{\sigma xx}\|^2
  + \frac{2\sigma}{3\mu} \|\tE_{\sigma xxx}\|^2 \nonumber \\
  &&\hspace{-1.1in} =\int_0^1 ({n^*} \tV_{\sigma x})_x\tJ_{\sigma x} {\rm d}x
    + \frac{2}{\mu\tau} \int_0^1 \tn_{\sigma x} \tE_{\sigma x} {\rm d}x
    + \int_0^1 g_{\sigma 1x} \tJ_{\sigma x} {\rm d}x  \nonumber \\
&&\hspace{-0.6in}
    + \frac{2}{3\mu} \int_0^1 (g_{\sigma 2}-\frac52 \tJ_{\sigma x})_x \tE_{\sigma x}
      {\rm d}x.
\end{eqnarray}
Differentiating $(\ref{4.1})_1, (\ref{4.1})_2$ and $(\ref{4.1})_3$
twice with respect to $x$, then testing the resultant equations by
$\mu\tn_{\sigma xx}, \tJ_{\sigma xx}$ and $\displaystyle
\frac{18\nu^2\gamma_1}{\epsilon^2\beta_1}\tE_{\sigma xx}$,
respectively, and integrating them over $[0,1]$, we get
\begin{equation}\label{new1}
 \frac{\mu}{2} \frac{\rm d}{{\rm d}t} \|\tn_{\sigma xx}\|^2
  + \mu \int_0^1 \tJ_{\sigma xxx} \tn_{\sigma xx} {\rm d}x
  = - \mu\nu \|\tn_{\sigma xxx}\|^2
    - \mu\sigma \|\tn_{\sigma xxxx}\|^2,
\end{equation}
\begin{eqnarray}
&&\hspace{-0.5in} \frac12 \frac{\rm d}{{\rm d}t} \big( \|\tJ_{\sigma xx}\|
  + \frac{\epsilon^2}{18}\|\tn_{\sigma xxx}\|^2 \big)
  + \frac{\epsilon^2\nu}{18} \|\tn_{\sigma xxxx}\|^2
  + \frac{\epsilon^2}{18} \sigma \|\tn_{\sigma xxxxx}\|^2
  + \frac{1}{\tau} \|\tJ_{\sigma xx}\|^2  \nonumber\\
&&\hspace{-0.1in} + \nu \|\tJ_{\sigma xxx}\|^2
  + \sigma \|\tJ_{\sigma xxxx}\|^2
  + \mu \int_0^1 \tn_{\sigma xxx} \tJ_{\sigma xx} {\rm d}x \nonumber\\
  &&\hspace{-0.7in} = - \frac23 \int_0^1 \tE_{\sigma xxx} \tJ_{\sigma xx} {\rm d}x
    + \int_0^1 ({n^*}\tV_{\sigma x})_{xx} \tJ_{\sigma xx} {\rm d}x
    + \int_0^1 g_{\sigma 1xx} \tJ_{\sigma xx} {\rm d}x, \label{new2}
\end{eqnarray}
and
\begin{eqnarray}
&&\hspace{-0.7in} \frac{9\nu^2\gamma_1}{\epsilon^2\beta_1}
    \frac{\rm d}{{\rm d}t} \|\tE_{\sigma xx}\|^2
  + \frac{36\nu^2\gamma_1}{\epsilon^2\tau\beta_1} \|\tE_{\sigma xx}\|^2
  + \frac{18\nu^3\gamma_1}{\epsilon^2\beta_1} \|\tE_{\sigma xxx}\|^2
  + \frac{18\nu^2\gamma_1}{\epsilon^2\beta_1}\sigma\|\tE_{\sigma xxxx}\|^2
    \nonumber \\
  &&\hspace{-0.9in}= \frac{54\nu^2\gamma_1}{\epsilon^2\tau\beta_1}
    \int_0^1 \tn_{\sigma xx}\tE_{\sigma xx} {\rm d}x
    - \frac{18\mu\nu^2\gamma_1}{\epsilon^2\beta_1}
      \int_0^1 \tJ_{\sigma xxx} \tE_{\sigma xx} {\rm d}x \nonumber \\
   &&\hspace{-0.7in}  + \frac{18\nu^2\gamma_1}{\epsilon^2\beta_1}
      \int_0^1 (g_{\sigma 2} - \frac52 \tJ_{\sigma x})_{xx}\tE_{\sigma xx} {\rm d}x. \label{new3}
\end{eqnarray}
Adding the  equalities (\ref{new1})--(\ref{new3}) gives
\begin{eqnarray}\label{4.7}
&&\hspace{-0.9in} \frac12 \frac{\rm d}{{\rm d}t} \big( \mu\|\tn_{\sigma xx}\|^2
  + \frac{\epsilon^2}{18} \|\tn_{\sigma xxx}\|^2
  + \|\tJ_{\sigma xx}\|^2
  + \frac{18\nu^2\gamma_1}{\epsilon^2\beta_1} \|\tE_{\sigma xx}\|^2
    \big) + \mu\nu \|\tn_{\sigma xxx}\|^2
 \nonumber \\
&& \hspace{-0.8in}   + (\frac{\epsilon^2\nu}{18}+\mu\sigma)\|\tn_{\sigma xxxx}\|^2 + \frac{\epsilon^2}{18}\sigma \|\tn_{\sigma xxxxx}\|^2
  + \frac{1}{\tau} \|\tJ_{\sigma xx}\|^2
  + \nu \|\tJ_{\sigma xxx}\|^2
  + \sigma \|\tJ_{\sigma xxxx}\|^2  \nonumber \\
&& + \frac{36\nu^2\gamma_1}{\epsilon^2\tau\beta_1}
    \|\tE_{\sigma xx}\|^2
  + \frac{18\nu^3\gamma_1}{\epsilon^2\beta_1} \|\tE_{\sigma xxx}\|^2
  + \frac{18\nu^2\gamma_1}{\epsilon^2\beta_1}\sigma
    \|\tE_{\sigma xxxx}\|^2   \nonumber \\
 &&\hspace{-1.1in} =\int_0^1 ({n^*}\tV_{\sigma x})_{xx}\tJ_{\sigma xx} {\rm d}x
    + \frac{54\nu^2\gamma_1}{\epsilon^2\tau\beta_1}
      \int_0^1 \tn_{\sigma xx} \tE_{\sigma xx} {\rm d}x
    - \frac23 \int_0^1 \tE_{\sigma xxx} \tJ_{\sigma xx} {\rm d}x
    \nonumber\\
&& \hspace{-0.9in}   - \frac{18\mu\nu^2\gamma_1}{\epsilon^2\beta_1}
      \int_0^1 \tJ_{\sigma xxx} \tE_{\sigma xx} {\rm d}x + \int_0^1 g_{\sigma 1xx} \tJ_{\sigma xx} {\rm d}x
    + \frac{18\nu^2\gamma_1}{\epsilon^2\beta_1}
      \int_0^1 (g_{\sigma 2}-\frac52 \tJ_{\sigma x})_{xx} \tE_{\sigma xx} {\rm d}x.
\end{eqnarray}

Utilizing the Sobolev embedding inequalities and skillful calculations we obtain
\begin{eqnarray}\label{4.8}
&&\hspace{-0.9in} \int_0^1 g_{\sigma 1} \tJ_{\sigma} {\rm d}x
  + \frac{2}{3\mu} \int_0^1 (g_{\sigma 2}-\frac52 \tJ_{\sigma x}) \tE_{\sigma} {\rm d}x
  + \int_0^1 g_{\sigma 1x} \tJ_{\sigma x} {\rm d}x
    \nonumber \\
&& \hspace{-0.8in} + \frac{2}{3\mu} \int_0^1 (g_{\sigma 2}-\frac52 \tJ_{\sigma x})_x \tE_{\sigma x} {\rm d}x + \int_0^1 g_{\sigma 1xx} \tJ_{\sigma xx} {\rm d}x
  + \frac{18\nu^2\gamma_1}{\epsilon^2\beta_1}
    \int_0^1 (g_{\sigma 2}-\frac52 \tJ_{\sigma x})_{xx} \tE_{\sigma xx} {\rm d}x
    \nonumber \\
&& \hspace{-1.1in} \leqslant
    C \big( \|\tn_{\sigma}\|_2^2 + \|\tJ_{\sigma}\|_2^2
            + \|\tE_{\sigma}\|_2^2
            + \|\tJ_{\sigma}\|_1 \|\tJ_{\sigma}\|_2^2
            + \|\tJ_{\sigma}\|_1^2 \|\tE_{\sigma}\|_2^2
            + \|\tJ_{\sigma}\|_1^2 \|\tV_{\sigma x}\|_2^2 \big)
              \nonumber \\
&&  \hspace{-0.9in}  + C \epsilon^2 \big( \|\tn_{\sigma x}\|_2^2
            + \|\tJ_{\sigma}\|_2^2 \|\tn_{\sigma xx}\|_1^2
            + \|\tJ_{\sigma}\|_1^2\|\tn_{\sigma x}\|_1^2\|\tn_{\sigma xx}\|_1^2
            + \|\tn_{\sigma x}\|_1 \|\tn_{\sigma xx}\|_1^2
            + \|\tJ_{\sigma}\|_2^2 \|\tn_{\sigma}\|_2^2 \big) \nonumber \\
&& \hspace{-0.9in}   + \frac{\mu\nu}{2} \|\tn_{\sigma x}\|_2^2
    + \frac{\nu\epsilon^2}{36} \|\tn_{\sigma xx}\|_2^2
    + \frac{1}{2\tau} \|\tJ_{\sigma}\|_2^2
    + \frac{\nu}{2} \|\tJ_{\sigma x}\|_2^2
    + \frac{2}{3\mu\tau} \|\tE_{\sigma}\|_1^2 \nonumber \\
&& \hspace{-0.9in}
    + \frac{\nu}{3\mu} \|\tE_{\sigma x}\|_1^2
    + \frac{12\nu^3\gamma_1}{\epsilon^2\beta_1} \|\tE_{\sigma
xxx}\|^2,
\end{eqnarray}
for some positive constant $C$.
In addition, it follows from $(\ref{4.1})_4$ that
\begin{eqnarray} \label{4.9}
 \|\tV_{\sigma}\|^2
  = \int_0^1 \big(\int_0^x \tV_{\sigma y}(y) {\rm d}y\big)^2 {\rm d}x
  \leqslant
    \|\tV_{\sigma x}\|^2
  \leqslant
    \|\tV_{\sigma xx}\|^2
  = \lambda^{-4} \|\tn_{\sigma}\|^2.
\end{eqnarray}

Now set
$$
 \Upsilon_k(t)
  := \mu \|\tn_{\sigma}\|_k^2
     + \frac{\epsilon^2}{18} \|\tn_{\sigma x}\|_k^2
     + \|\tJ_{\sigma}\|_k^2
     + \|\tilde{e}_{\sigma}\|_k^2$$
     with
$$ \|\tilde{e}_{\sigma}\|_k^2 =
 \left\{
  \begin{array}{ll}
   \frac{2}{3\mu}\|\tE_{\sigma}\|_k^2, \ k=0,1, \medskip\\
   \frac{2}{3\mu}\|\tE_{\sigma}\|_1^2
    + \frac{18\nu^2\gamma_1}{\epsilon^2\beta_1}
      \|\tE_{\sigma xx}\|^2, \ k=2.
  \end{array}
 \right.
$$
Then, with the aid of (\ref{4.9}), we can deduce from
(\ref{4.5})--(\ref{4.8}) that there exists a constant $C$ independent
of $\sigma$ such that
\begin{eqnarray}\label{4.10}
&&  \hspace{-0.9in} \frac{\rm d}{{\rm d}t} \Upsilon_2(t)
  + \frac{\mu\nu}{2} \|\tn_{\sigma x}\|_2^2
  + (\frac{\nu\epsilon^2}{36}+\mu\sigma)\|\tn_{\sigma xx}\|_1^2
  + (\frac{\epsilon^2\nu}{54}+\mu\sigma)\|\tn_{\sigma xxxx}\|^2
  + \frac{\epsilon^2}{18}\sigma \|\tn_{\sigma xxx}\|_2^2 \nonumber \\
&& \hspace{-0.9in} + \frac{1}{2\tau} \|\tJ_{\sigma}\|_2^2
  + \frac{\nu}{2} \|\tJ_{\sigma x}\|_2^2
  + \sigma \|\tJ_{\sigma xx}\|_2^2
  + \frac{2}{3\mu\tau} \|\tE_{\sigma}\|_1^2
  + \frac{\nu}{3\mu} \|\tE_{\sigma x}\|_1^2
  + \frac{2\sigma}{3\mu} \|\tE_{\sigma xx}\|_1^2 \nonumber \\
&& \hspace{-0.9in}  + \frac{18\nu^2\gamma_1}{\epsilon^2\tau\beta_1}\|\tE_{\sigma xx}\|^2
  + \frac{6\nu^3\gamma_1}{\epsilon^2\beta_1} \|\tE_{\sigma xxx}\|^2
  + \frac{18\nu^2\gamma_1}{\epsilon^2\beta_1} \sigma
    \|\tE_{\sigma xxxx}\|^2  \nonumber \\
 &&  \hspace{-1.1in}  \leqslant
    C \Big( \Upsilon_2(t) + \Upsilon_2^2(t)
            + \Upsilon_1(t) \Upsilon_2(t)
            + \Upsilon_1^2(t) \Upsilon_2(t) \Big).
\end{eqnarray}
It then follows that there exists a time $0 < t_1^* \leqslant
t_{\sigma}$ such that $\Upsilon_2(\cdot) \in
\mathcal{C}^1([0,t_1^*])$ and $\Upsilon_2(t) \leqslant
2\Upsilon_2(0)$ for $0 \leqslant t \leqslant t_1^*$. Moreover, for
$0\leqslant t'\leqslant t'' \leqslant t_{\sigma}$, we have
\begin{eqnarray}\label{+1}
\hspace{0.9in} \big|\Upsilon_2(t'')-\Upsilon_2(t')\big|
 \leqslant C |t''-t'|,
\end{eqnarray}
where $C$ does not depend on $\sigma$. The estimate (\ref{+1}) gives
us an estimate from below for the earliest time $t_2^*>0$ at which
the solution $(\tn_{\sigma}, \tJ_{\sigma}, \tE_{\sigma})$ is able to
violate the conditions (\ref{4.3}).

Let $t^*:=\min\big\{t_1^*, t_2^*\big\}$, we can also conclude from
(\ref{4.1}) and (\ref{4.10}) that there exists a constant $C$
independent of $\sigma$ such that
\begin{eqnarray}\label{4.11}
 \hspace{-0.9in}\left \{
  \begin{array}{ll}
   \|\tn_{\sigma}\|_{L^{\infty}(0,t^*; H^3)}
    + \|\tn_{\sigma}\|_{L^2(0,t^*; H^4)}
    \leqslant C, \quad
   \|\tV_{\sigma xx}\|_{L^{\infty}(0,t^*; H^3)}
    + \|\tV_{\sigma xx}\|_{L^2(0,t^*; H^4)}
    \leqslant C, \medskip \\
   \|\tJ_{\sigma}\|_{L^{\infty}(0,t^*; H^2)}
    + \|\tJ_{\sigma}\|_{L^2(0,t^*; H^3)}
    \leqslant C, \quad
   \|\tE_{\sigma}\|_{L^{\infty}(0,t^*; H^2)}
    + \|\tE_{\sigma}\|_{L^2(0,t^*; H^3)}
    \leqslant C, \medskip \\
   \|\partial_t \tn_{\sigma}\|_{L^{\infty}(0,t^*; H^{-1})}
    + \|\partial_t \tJ_{\sigma}\|_{L^{\infty}(0,t^*; H^{-2})}
    + \|\partial_t \tE_{\sigma}\|_{L^{\infty}(0,t^*; H^{-2})}
    \leqslant C.
  \end{array}
 \right.
\end{eqnarray}

\subsection{The convergence of a subsequence of
$\{(\tn_{\sigma}, \tJ_{\sigma}, \tE_{\sigma}, \tV_{\sigma})\}_{\sigma}$
as $\sigma \rightarrow 0$. } \label{sec:4.12}

Since the embedding $H^{k+1}(0,1)\hookrightarrow H^k(0,1),
k=0,1,2,3$ is compact, there exist a subsequence of $\{\tn_{\sigma}\}_{\sigma},
\{\tJ_{\sigma}\}_{\sigma}, \{\tE_{\sigma}\}_{\sigma},
\{\tV_{\sigma}\}_{\sigma}$, still denoted by the
same symbols,  and $\tn, \tJ, \tE, \tV$ such that as $\sigma
\rightarrow 0$,
$$  \begin{array}{ll}
   \tn_{\sigma} \rightharpoonup^* \tn \ \mbox{weakly star in} \
    L^{\infty}(0,t^*; H^3), \medskip\\
   \partial_t \tn_{\sigma} \rightharpoonup^* \partial_t \tn \
    \mbox{weakly star in} \ L^{\infty}(0,t^*; H^{-1}), \medskip \\
   \tn_{\sigma} \rightharpoonup \tn \ \mbox{weakly in} \
    L^2(0,t^*; H^4),  \medskip \\
   \tJ_{\sigma} \rightharpoonup^* \tJ \ \mbox{weakly star in} \
    L^{\infty}(0,t^*; H^2), \medskip \\
   \partial_t \tJ_{\sigma} \rightharpoonup^* \partial_t \tJ \
    \mbox{weakly star in} \ L^{\infty}(0,t^*; H^{-2}), \medskip \\
   \tJ_{\sigma} \rightharpoonup \tJ \ \mbox{weakly in} \
    L^2(0,t^*; H^3),  \medskip \\
   \tE_{\sigma} \rightharpoonup^* \tE \ \mbox{weakly star in} \
    L^{\infty}(0,t^*; H^2), \medskip \\
   \partial_t \tE_{\sigma} \rightharpoonup^* \partial_t \tE \
    \mbox{weakly star in} \ L^{\infty}(0,t^*; H^{-2}), \medskip \\
   \tE_{\sigma} \rightharpoonup \tE \ \mbox{weakly in} \
    L^2(0,t^*; H^3),   \end{array}$$

    $$\begin{array}{ll}
   \tn_{\sigma} \rightarrow \tn \ \mbox{strongly in} \
    \mathcal{C}([0,t^*]; H^{3-\iota}) \ \mbox{and} \
    L^2(0,t^*; H^{4-\iota}), \ \ \iota > 0, \medskip \\
   \tJ_{\sigma} \rightarrow \tJ \ \mbox{strongly in} \
    \mathcal{C}([0,t^*]; H^{2-\iota}) \ \mbox{and} \
    L^2(0,t^*; H^{3-\iota}), \ \ \iota > 0, \medskip \\
   \tE_{\sigma} \rightarrow \tE \ \mbox{strongly in} \
    \mathcal{C}([0,t^*]; H^{2-\iota}) \ \mbox{and} \
    L^2(0,t^*; H^{3-\iota}), \ \ \iota > 0, \medskip \\
   \tV_{\sigma x} \rightarrow \tV_x \ \mbox{strongly in} \
    \mathcal{C}([0,t^*]; H^{4-\iota}) \ \mbox{and} \
    L^2(0,t^*; H^{5-\iota}), \ \ \iota > 0.
  \end{array}$$

\subsection{Existence of local solutions}
In this subsection, we  show that the limit $(\tn, \tJ, \tE, \tV)$
of $(\tn_{\sigma}, \tJ_{\sigma}, \tE_{\sigma}, \tV_{\sigma})$ in the
previous subsection is in fact a local solution of IBVP
(\ref{3.2})--(\ref{3.3}).

Based on the convergence shown in Section \ref{sec:4.12}, it is easy
to check that
\begin{eqnarray*}
 - \lambda^2 \tV_{xx} = - \tn, \quad \forall (t,x)
   \in [0,t^*] \times [0,1].
\end{eqnarray*}
Let $\phi \in H_0^2(0,1)$ be a test function.  Then due to
Sect. \ref{sec:4.12}, for any $t \in [0, t^*]$, we have
\begin{eqnarray*}
&&\hspace{-0.8in} \left| \int_0^t\int_0^1(g_{\sigma 1}-g_1)\phi{\rm d}x{\rm d}t\right|
    \nonumber \\
  && \hspace{-1in} = \left| \frac23 \int_0^t \int_0^1
    \left(\frac{(J^*+\tJ_{\sigma})^2}{{n^*}+\tn_{\sigma}} - \frac{(J^*)^2}{{n^*}}
         - \frac{(J^*+\tJ)^2}{{n^*}+\tn} + \frac{{J^*}^2}{{n^*}} \right)
         \phi_x  {\rm d}x {\rm d}t  \right. \\
&&\hspace{-0.8in}  - \frac{\epsilon^2}{18} \int_0^t \int_0^1
      \left( \frac{\big( ({n^*}+\tn_{\sigma})_x^2 \big)_x}{{n^*}+\tn_{\sigma}}
            - \frac{\big( ({n^*}+\tn)_x^2 \big)_x}{{n^*}+\tn}
            - \frac{({n^*}+\tn_{\sigma})_x^3}{{n^*}+\tn_{\sigma}}
            + \frac{({n^*}+\tn)_x^3}{{n^*}+\tn} \right) \phi
              {\rm d}x {\rm d}t  \\
&&\hspace{-0.8in}   +  \left.\int_0^t \int_0^1 (V^*_{x}\tn_{\sigma} - V^*_{x}\tn) \phi
      {\rm d}x {\rm d}t
    + \int_0^t \int_0^1 (\tn_{\sigma}\tV_{\sigma x} - \tn \tV_x) \phi
      {\rm d}x {\rm d}t  \right|  \\
&& \hspace{-1in} \rightarrow  \,0 \quad \mbox{as}\quad \sigma
\rightarrow 0,
\end{eqnarray*}
and
$$
 \sigma \left| \int_0^t \int_0^1 \tJ_{\sigma xxxx} \phi
                {\rm d}x {\rm d}t \right|
   = \sigma \left| \int_0^t \int_0^1 \tJ_{\sigma xx} \phi_{xx}
                {\rm d}x {\rm d}t \right| \rightarrow 0 \quad \mbox{as}\quad \sigma
   \rightarrow 0.
$$

In addition, for any $t\, \in [0, t^*]$, it can be shown that as
$\sigma \rightarrow 0$
$$
  \begin{array}{ll}
   \displaystyle \int_0^t \int_0^1 (\tJ_{\sigma t} - \tJ_t) \phi {\rm d}x {\rm d}t
   \rightarrow 0,   \qquad
   - \frac{\epsilon^2}{18} \int_0^t \int_0^1
     (\tn_{\sigma xxx} - \tn_{xxx}) \phi {\rm d}x {\rm d}t
   \rightarrow 0,  \medskip \\
   \displaystyle \frac23\int_0^t\int_0^1(\tE_{\sigma x}-\tE_x)\phi{\rm d}x {\rm d}t
   \rightarrow 0,   \quad
   - \int_0^t \int_0^1 ({n^*} \tV_{\sigma x} - {n^*} \tV_x) \phi
     {\rm d}x {\rm d}t
   \rightarrow 0,  \medskip \\
   \displaystyle \mu \int_0^t \int_0^1 (\tn_{\sigma x} - \tn_x) \phi
    {\rm d}x {\rm d}t  \rightarrow 0,  \quad
   \int_0^t \int_0^1 \big[(\nu \tJ_{\sigma xx} - \frac{1}{\tau}\tJ_{\sigma})
     - (\nu \tJ_{xx} - \frac{1}{\tau} \tJ) \big] \phi {\rm d}x {\rm d}t
   \rightarrow 0.
  \end{array}
$$
Thus, taking the limit $\sigma \rightarrow 0$ gives
$$
 \int_0^t \int_0^1 (\tJ_t - \frac{\epsilon^2}{18}\tn_{xxx}
   + \frac23 \tE_x - {n^*} \tV_x + \mu\tn_x - g_1) \phi {\rm d}x{\rm d}t
 = \int_0^t \int_0^1 (\nu \tJ_{xx} - \frac{1}{\tau}\tJ){\rm d}x{\rm d}t \quad \forall \,\, t \in [0,t^*].
$$

Similarly, we can also get
\begin{eqnarray*}
 &&\int_0^t \int_0^1 (\tn_t + \tJ_x) \phi {\rm d}x{\rm d}t
 = \nu \int_0^t \int_0^1 \tn_{xx} \phi {\rm d}x{\rm d}t \\
 &&\int_0^t \int_0^1 [\tE_t - \frac{3}{\tau}\tn + (\mu+\frac52)\tJ_x
                    - g_2] \phi {\rm d}x{\rm d}t
 = \int_0^t\int_0^1 (\nu \tE_{xx}-\frac{2}{\tau} \tE)\phi{\rm d}x{\rm d}t.
\end{eqnarray*}
Therefore, $(\tn,\tJ,\tE,\tV)$ is a local solution of the IBVP
(\ref{3.2})--(\ref{3.3}) on $[0,t^*]$  with
$$
  \begin{array}{ll}
   \tn \in H^1(0,t^*; H^2(0,1)) \cap L^2(0,t^*; H^4(0,1)), \medskip\\
   \tV \in H^1(0,t^*; H^4(0,1)) \cap L^2(0,t^*; H^6(0,1)), \medskip\\
   \tJ \in H^1(0,t^*; H^1(0,1)) \cap L^2(0,t^*; H^3(0,1)), \medskip\\
   \tE \in H^1(0,t^*; H^1(0,1)) \cap L^2(0,t^*; H^3(0,1)).
  \end{array}
$$

\subsection{Uniqueness of solutions}

Let $(\tn_1, \tJ_1, \tE_1, \tV_1)$ and $(\tn_2, \tJ_2, \tE_2, \tV_2)$ be two
solutions of IBVP (\ref{3.2})--(\ref{3.3}) with the same initial and
boundary conditions.   Set
\begin{eqnarray*}
 \hat{n} = \tn_1 - \tn_2, \quad
 \hat{J} = \tJ_1 - \tJ_2, \quad
 \hat{\mathcal{E}} = \tE_1 - \tE_2, \quad
 \hat{V} = \tV_1 - \tV_2.
\end{eqnarray*}
It then follows from (\ref{3.2}) that
\begin{eqnarray}\label{4.13}
 \left\{
  \begin{array}{ll}
   \hat{n}_t - \nu\hat{n}_{xx} + \hat{J}_x = 0,
    \medskip \\
   \displaystyle \hat{J}_t - \nu\hat{J}_{xx} + \frac{1}{\tau}\hat{J}
    - \frac{\epsilon^2}{18} \hat{n}_{xxx}
    + \frac23 \hat{\mathcal{E}}_x
    - {n^*} \hat{V}_x
    + \mu \hat{n}_x
   = \hat{g}_1, \medskip \\
   \displaystyle \hat{\mathcal{E}}_t - \nu\hat{\mathcal{E}}_{xx} + \frac{2}{\tau}\hat{\mathcal{E}}
    - \frac{3}{\tau} \hat{n}
    + \mu \hat{J}_x
   = \hat{g}_2, \medskip \\
   - \lambda^2 \hat{V}_{xx}
   = - \hat{n},
  \end{array}
 \right.
\end{eqnarray}
where $\hat{g}_1$ and $\hat{g}_2$ are listed in \ref{app:para}.

Testing $(\ref{4.13})_2, (\ref{4.13})_3$ and $(\ref{4.13})_4$ by
$\hat{J}$, $\xi$ (to be specified later), $\hat{\mathcal{E}}$ and
$\hat{V}$, respectively, then integrating the resultant equations
over $[0,1]$, this gives
\begin{eqnarray}
&&\hspace{-0.8in} \frac12 \frac{\rm d}{{\rm d}t} \|\hat{J}\|^2
  + \nu \|\hat{J}_x\|^2
  + \frac{1}{\tau} \|\hat{J}\|^2
  - \frac{\epsilon^2}{18} \int_0^1 \hat{n}_{xxx}\hat{J}{\rm d}x
  + \mu \int_0^1 \hat{n}_x \tJ {\rm d}x \nonumber \\
  &&\hspace{-1in}= \frac12 \frac{\rm d}{{\rm d}t} \left(\|\hat{J}\|^2
      + \frac{\epsilon^2}{18}\|\hat{n}_x\|^2
      + \mu \|{\tn}\|^2 \right)
    + \nu\|\hat{J}_x\|^2
    + \frac{1}{\tau} \|\hat{J}\|^2
    + \mu\nu \|\hat{n}_x\|^2
    + \frac{\epsilon^2\nu}{18} \|\hat{n}_{xx}\|^2 \nonumber \\
  &&\hspace{-1in}=  \int_0^1 (\hat{g}_1 - \frac23\hat{\mathcal{E}}_x + {n^*}\hat{V}_x) \hat{J}
    {\rm d}x,  \label{4.14} \\
&&\hspace{-0.8in}  \frac{\xi}{2} \frac{\rm d}{{\rm d}t} \|\hat{\mathcal{E}}\|^2
  + \frac{2\xi}{\tau} \|\hat{\mathcal{E}}\|^2
  + \xi\nu \|\hat{\mathcal{E}}_x\|^2
  = \xi \int_0^1 \big( \hat{g}_2 - \mu \hat{J}_x
    + \frac{3}{\tau}\hat{n} \big) \hat{\mathcal{E}} {\rm d}x, \label{4.15}
\end{eqnarray}
and
\begin{eqnarray*}
\lambda^2 \|\hat{V}_x\|^2
 = \int_0^1 \hat{n} \hat{V} {\rm d}x
 \leqslant
   \|\hat{n}\| \|\hat{V}\|,
\end{eqnarray*}
which, along with the Poincar\'e inequality, implies that
$\|\hat{V}_x\| \leqslant C \|\hat{n}\|$ for some positive constant
$C$.

By the definitions of $\hat{g}_1$
and $\hat{g}_2$,  the right-hand sides of the
equalities (\ref{4.14}) and (\ref{4.15}) can be estimated to satisfy, respectively,
\begin{eqnarray}
&&\hspace{-0.8in} \left| \int_0^1 (\hat{g}_1 - \frac23\hat{\mathcal{E}}_x
                 + {n^*}\hat{V}_x) \hat{J} {\rm d}x \right| \nonumber \\
 &&\hspace{-1in} \leqslant
  C \Big[ (1+\|{\tn}_i\|_{L^{\infty}}+\|{\tn}_i\|_{L^{\infty}}^2) \|\hat{n}\|
          + (1+\|{\tn}_i\|_{L^{\infty}})(1+\|{\tn}_i\|_{L^{\infty}})\|\hat{J}\|
    \Big] \|\hat{J}_x\|  \nonumber \\
&& \hspace{-0.8in} + C\epsilon^2 \Big\{
     \Big( \|\hat{n}\|_2 + \|{\tn}_{ix}\|_{L^{\infty}}^2 \|\hat{n}_x\|
       +(\|{\tn}_i\|_{L^{\infty}}+\|{\tn}_{ix}\|_{L^{\infty}})\|\hat{n}_x\|_1
     \Big) \|\hat{J}\| \nonumber \\
&&\hspace{-0.3in} + \Big( (1+\|{\tn}_i\|_{L^{\infty}})\|{\tn}_{ix}\|_{L^{\infty}}
              \|\hat{n}_x\|
              + \big( \|{\tn}_{ix}\|_{L^{\infty}}
                      + \|{\tn}_{ix}\|_{L^{\infty}}^2 \big) \|\hat{n}\|
       \Big) \|\hat{J}_x\| \nonumber \\
&&\hspace{-0.3in} + (1+\|{\tn}_{ix}\|_{L^{\infty}}+\|{\tn}_{ix}\|_{L^{\infty}}^2)
             \|{\tn}_i\|_{L^{\infty}} \|\hat{J}\| \|\hat{J}_x\|  \nonumber \\
            && \hspace{-0.3in} + (1+\|{\tn}_{ix}\|_{L^{\infty}}+\|{\tn}_{ix}\|_{L^{\infty}}^2
                +\|{\tn}_{ix}\|_{L^{\infty}}^3) \|\hat{n}\| \|\hat{J}\|
      \Big\} \nonumber \\
&&  \hspace{-0.8in}   + \|V^*_{x}\|_{L^{\infty}} \|\hat{n}\| \|\hat{J}\|
     + \big( \|{\tn}_1\|_{L^{\infty}} \|\hat{V}_x\|^2
             + \|{\tV}_{2x}\|_{L^{\infty}} \|\hat{n}\| \big)
            \|\hat{J}\|, \ \ i=1,2, \label{4.16}\end{eqnarray}
            and
            \begin{eqnarray}
&&\hspace{-0.8in}  \left| \xi \int_0^1 ( \hat{g}_2 - \mu \hat{J}_x
        + \frac{3}{\tau}\hat{n}) \hat{\mathcal{E}} {\rm d}x \right|
          \nonumber \\
 &&\hspace{-1in}  \leqslant
   C \xi \left( (1+\|{\tn}_i\|_{L^{\infty}}) \|\hat{\mathcal{E}}\|
     + (1+\|\tE_i\|_{L^{\infty}}) \|\hat{J}\|
     + (\|{\tn}_i\|_{L^{\infty}} + \|{\tn}_i\|_{L^{\infty}}^2) \|\hat{J}\|
     \right) \|\hat{\mathcal{E}}_x\| \nonumber \\
&& \hspace{-0.8in}  + C \xi\epsilon^2 \Big[
     (1+\|{\tn}_{ix}\|_{L^{\infty}})(\|\hat{n}_x\| + \|\hat{J}\|)
     + (1+\|{\tn}_{ix}\|_{L^{\infty}})\|{\tn}_i\|_{L^{\infty}} \|\hat{n}_x\|
       \nonumber \\
&&  \hspace{-0.2in}  + (1+\|{\tn}_i\|_{L^{\infty}})\|\hat{n}_{xx}\|
     + \|{\tn}_{ixx}\| \|\hat{J}\|_{L^{\infty}} \Big]
     \|\hat{\mathcal{E}}_x\| \nonumber \\
&&  \hspace{-0.8in}   + C \xi \big( \|\hat{n}\| + \|\hat{J}\|
     + \|{\tV}_{ix}\|_{L^{\infty}} \|\hat{J}\|
     + (1+\|{\tn}_i\|_{L^{\infty}}) \|\hat{V}_x\| \big) \|\hat{\mathcal{E}}\|,
        \ \ i=1,2 \label{4.17}
\end{eqnarray}
for some positive constant $C$.

Finally, taking (\ref{4.14})--(\ref{4.17}) into account, choosing an
appropriate value for the constant $\xi$ and applying Young's
inequality, we can obtain
$$
 \frac{\rm d}{{\rm d}t}
  \big( \mu\|\hat{n}\|^2
        + \frac{\epsilon^2}{18} \|\hat{n}_x\|^2
        + \|\hat{J}\|^2
        + \xi \|\hat{\mathcal{E}}\|^2 \big)
  \leqslant
   C \big( \|\hat{n}\|^2 + \epsilon^2 \|\hat{n}_x\|^2
           + \|\hat{J}\|^2 + \xi \|\hat{\mathcal{E}}\|^2 \big).
$$
It then follows from Gronwall's inequality that
\begin{eqnarray*}
 \hat{n} = 0, \quad \hat{J} = 0, \quad \hat{\mathcal{E}} = 0, \quad \hat{V} = 0.
\end{eqnarray*}
The proof of Theorem \ref{T1.3} is thus complete.
\section{Stability of the steady-state solution} \label{sec:stab}

In this section, we show the stability of the steady-state solution
$(n^*, J^*, \mathcal{E}^*, V^*)$. More precisely, we prove the
solution $(\tn(x,t), \tJ(x,t), \tE(x,t), \tV(x,t))$  of the IBVP
(\ref{3.2})--(\ref{3.3}) decays to 0 exponentially as $t \to
\infty$.   For this purpose, we first show the global existence of
solutions, i.e., the local solution established in Section
\ref{sec:ls} exists globally in time.   Let $(n(x,t), J(x, t),
\mathcal{E}(x, t), V(x,t))$ be the local solution to the IBVP
(\ref{1.2})--(\ref{1.4}) shown in Theorem \ref{T1.3} that exists on
$t \in (0, t^*)$ and set
\begin{eqnarray}\label{3.4}
 \delta_{t^*} := \max\limits_{0\leqslant t \leqslant t^*}
             (\|\tn(t)\|_3^2+\|\tJ(t)\|_2^2+\|\tE(t)\|_2^2).
\end{eqnarray}
The following lemma gives an upper bound of the solution $(\tn(x,t),
\tJ(x,t), \tE(x,t), \tV(x,t))$ of  the IBVP
(\ref{3.2})--(\ref{3.3}).
\begin{lem}\label{L3.1}~
Let Assumptions $\mathbf{(A1)}-\mathbf{(A2)}$ and the condition
(\ref{newcond1}) hold.  Then the local solution $(\tn(x,t),
\tJ(x,t), \tE(x,t), \tV(x,t))$ of the IBVP (\ref{3.2})--(\ref{3.3})
satisfies
\begin{eqnarray}\label{3.5}
 & \|{\tn}(t)\|_3^2 + \|{\tJ}(t)\|_2^2 + \|{\tE}(t)\|_2^2 + \|{\tV}(t)\|_5^2
    \nonumber \\
 & + \varsigma_2 \int_0^t
    e^{\varsigma_1(s-t)}(\|{\tn}(s)\|_4^2+\|{\tJ}(s)\|_3^2+\|{\tE}(s)\|_3^2+\|{\tn}_s(s)\|_2^2)
    {\rm d}s  \nonumber \\
  \leqslant &
   e^{-\varsigma_1 t}
   (\|{\tn}_0\|_3^2+\|{\tJ}_0\|_2^2+\|{\tE}_0\|_2^2+\|{\tV}_0\|_3^2)
\end{eqnarray}
with some positive constants $\varsigma_1$ and $\varsigma_2$,
provided $\delta_0+\delta_{t^*}+J_b$ is sufficiently small.
   \end{lem}
\begin{proof}
First, by (\ref{3.4}) and the Sobolev inequality, it is easy to see that
if $\delta_{t^*}$ is sufficiently small, then there exist constants
$n_-, n_+, J_-$ and $J_+$ such that
\begin{eqnarray*}
 0 < n_- \leqslant n^*+{\tn} \leqslant n_+, \quad
 J_- \leqslant {J^*}+{\tJ} \leqslant J_+.
\end{eqnarray*}
In addition, it follows from Assumption {\bf (A1)} and (\ref{1.8})
that there exists a positive constant $C$ such that
\begin{eqnarray}\label{3.6}
\left\{ \begin{array}{ll}
 \|n^*_{x}\|_2^2 + \|J^*_{x}\|_1^2 + \|\mathcal{E}^*_{x}\|_1^2
  + \|V^*\|_5^2
 \leqslant
 C\delta_0^2, \medskip \\
 \|{J^*}\|_{L^{\infty}(0,1)}
  \leqslant
C \delta_0 + J_b.
 \end{array}
\right.
\end{eqnarray}

Now, testing $(\ref{3.2})_1$ by $\mu{\tn}$ and $(\ref{3.2})_2$ by ${\tJ}$
and integrating them over $[0,1]$, then adding the resultant
equalities, we have
\begin{eqnarray}\label{3.8}
&& \hspace{-1in}\frac12 \frac{\rm d}{{\rm d}t} (\mu \|{\tn}\|^2 + \|{\tJ}\|^2)
  = - \mu\nu \|{\tn}_x\|^2
 - \nu \|{\tJ}_x\|^2
  - \frac{1}{\tau} \|{\tJ}\|^2
  + \frac{\epsilon^2}{18} \int_0^1 {\tn}_{xxx}{\tJ} {\rm d}x \nonumber \\
&& \hspace{0.5in}- \frac23 \int_0^1 {\tE}_x {\tJ} {\rm d}x
  + \int_0^1 n^* {\tV}_x {\tJ} {\rm d}x
  + \int_0^1 g_1 {\tJ} {\rm d}x,
\end{eqnarray}
where $g_1$ is given in \ref{app:para}.  We next estimate each term
on the right-hand side of (\ref{3.8}).

First, due to $(\ref{3.2})_1$ and the boundary conditions we have
\begin{eqnarray}\label{3.9}
 \hspace{-0.5in} -\frac{\epsilon^2}{18} \int_0^1 {\tn}_{xxx}{\tJ} {\rm d}x
 = \frac{\epsilon^2}{18}\int_0^1{\tn}_{xx}(\nu{\tn}_{xx}-{\tn}_t){\rm d}x
 = \frac{\epsilon^2}{18} \nu\|{\tn}_{xx}\|^2
   + \frac{\epsilon^2}{36} \frac{\rm d}{{\rm d}t}\|{\tn}_x\|^2.
\end{eqnarray}
Similarly, it follows from $(\ref{3.2})_1$ and $(\ref{3.2})_4$ that
\begin{eqnarray}\label{3.10}
 &&\hspace{-1in} -\int_0^1 n^* {\tV}_x {\tJ} {\rm d}x
 = \int_0^1 n^* {\tV} {\tJ}_x {\rm d}x
   + \int_0^1 n^*_{x} {\tV} {\tJ} {\rm d}x
 = \int_0^1 n^* {\tV} (\nu{\tn}_{xx}-{\tn}_t) {\rm d}x
   + \int_0^1 n^*_{x} {\tV} {\tJ} {\rm d}x  \nonumber \\
  && \hspace{-1in} = \frac{\nu}{\lambda^2} \int_0^1 n^*{\tn}^2 {\rm d}x
   + \nu \int_0^1 n^*_{x} {\tV}_x {\tn} {\rm d}x
   - \nu \int_0^1 n^*_{x} {\tV} {\tn}_x {\rm d}x
   - \int_0^1 n^* {\tV} {\tn}_t {\rm d}x
   + \int_0^1 n^*_{x} {\tV} {\tJ} {\rm d}x  \nonumber \\
  && \hspace{-1in}= \frac{\nu}{\lambda^2} \int_0^1 n^* {\tn}^2 {\rm d}x
   + \nu \int_0^1 n^*_{x} {\tV}_x {\tn} {\rm d}x
   - \nu \int_0^1 n^*_{x} {\tV} {\tn}_x {\rm d}x
   - \lambda^2 \int_0^1 n^* {\tV} {\tV}_{xxt} {\rm d}x
   + \int_0^1 n^*_{x} {\tV} {\tJ} {\rm d}x  \nonumber \\
 && \hspace{-1in} = \frac{\nu}{\lambda^2} \int_0^1 n^* {\tn}^2 {\rm d}x
   + \frac{\lambda^2}{2} \frac{\rm d}{{\rm d}t} \int_0^1 n^* {\tV}_x^2
     {\rm d}x
   + \lambda^2 \int_0^1 n^*_{x} {\tV} {\tV}_{xt} {\rm d}x
   + \nu \int_0^1 n^*_{x} {\tV}_x {\tn} {\rm d}x  \nonumber \\
 && \hspace{-0.7in} - \nu \int_0^1 n^*_{x} {\tV} {\tn}_x {\rm d}x
   + \int_0^1 n^*_{x} {\tV} {\tJ} {\rm d}x.
\end{eqnarray}
For the last term on the right-hand of (\ref{3.8}), we can derive
from (\ref{3.4}) and (\ref{3.6}) that
\begin{eqnarray}\label{3.11}
 \int_0^1 g_1 {\tJ} {\rm d}x
  \leqslant
   C (\delta_0+\delta_{t^*}+J_b) (\|{\tJ}\|_1^2 + \|{\tn}\|^2 + \epsilon^2\|{\tn}_{xx}\|^2
     + \|{\tV}\|_1^2),
\end{eqnarray}
for some constant $C>0$.

Inserting (\ref{3.9})--(\ref{3.11}) into (\ref{3.8})  gives
\begin{eqnarray}\label{3.12}
&&\hspace{-0.9in} \frac12 \frac{\rm d}{{\rm d}t} \big(
  \mu\|{\tn}\|^2 + \frac{\epsilon^2}{18}\|{\tn}_x\|^2 + \|{\tJ}\|^2
   + \lambda^2 \int_0^1 {n^*}{\tV}_x^2 {\rm d}x \big)
 + \frac{\epsilon^2}{18}\nu \|{\tn}_{xx}\|^2
 + \mu\nu \|{\tn}_x\|^2
 + \nu \|{\tJ}_x\|^2  \nonumber \\
&&\hspace{-0.9in} + \frac{1}{\tau} \|{\tJ}\|^2
 + \frac{\nu}{\lambda^2} \int_0^1 {n^*}{\tn}^2 {\rm d}x
 + \frac23 \int_0^1 {\tE}_x {\tJ} {\rm d}x \nonumber \\
  &&\hspace{-1in} =  \int_0^1 g_1 {\tJ} {\rm d}x
   - \lambda^2 \int_0^1 n^*_{x}{\tV}{\tV}_{xt} {\rm d}x
   - \nu \int_0^1 n^*_{x}{\tV}_x {\tn} {\rm d}x
   + \nu \int_0^1 n^*_{x}{\tV} {\tn}_x {\rm d}x
   - \int_0^1 n^*_{x}{\tV} {\tJ} {\rm d}x  \nonumber \\
 && \hspace{-1in} \leqslant
  C (\delta_0+\delta_{t^*}+J_b) \big(\|{\tn}\|_1^2 + \epsilon^2\|{\tn}_{xx}\|^2
          + \|{\tJ}\|_1^2 + \|{\tV}\|_1^2 + \|{\tV}_{xt}\|^2 \big),
\end{eqnarray}
for some constant $C> 0$.

Next, multiplying $(\ref{3.2})_3$ by $\displaystyle \frac{4}{3(2\mu+5)}{\tE}$
and integrating the resultant equality, we obtain
\begin{eqnarray}\label{3.13}
&&\hspace{-1in} \frac{2}{3(2\mu+5)} \frac{\rm d}{{\rm d}t} \|{\tE}\|^2
  + \frac{8}{3\tau(2\mu+5)} \|{\tE}\|^2
  + \frac{4\nu}{3(2\mu+5)} \|{\tE}_x\|^2
  + \frac23 \int_0^1 {\tE} {\tJ}_x {\rm d}x
  \nonumber \\
  &&= \frac{4}{\tau(2\mu+5)} \int_0^1 {\tn} {\tE} {\rm d}x + \frac{4}{3(2\mu+5)} \int_0^1 g_2 {\tE} {\rm d}x,
\end{eqnarray}
where $g_2$ is listed in \ref{app:para}.
Similar to (\ref{3.11}), we can obtain
\begin{eqnarray}\label{3.14}
\hspace{-0.5in} \frac{4}{3(2\mu+5)} \int_0^1 g_2 {\tE} {\rm d}x
  \leqslant
   C(\delta_0+\delta_{t^*}+J_b)
    (\|{\tn}\|_1^2 + \epsilon^2 \|{\tn}_{xx}\|^2 + \|{\tJ}\|^2 + \|{\tE}\|_1^2).
\end{eqnarray}
In addition, observe that
\begin{eqnarray*}
 \|{\tn}\|^2
  = \int_0^1 (\int_0^x {\tn}_y(y) {\rm d}y)^2 {\rm d}x
  \leqslant
    \int_0^1 {\tn}_y^2 {\rm d}y
  = \|{\tn}_x\|^2, \medskip \\
 \|{\tE}\|^2
  = \int_0^1 (\int_0^x {\tE}_y(y) {\rm d}y)^2 {\rm d}x
  \leqslant
    \int_0^1 {\tE}_y^2 {\rm d}y
  = \|{\tE}_x\|^2,
\end{eqnarray*}
and
\begin{eqnarray*}
 \frac{4}{\tau(2\mu+5)} \int_0^1 {\tn} {\tE} {\rm d}x
  \leqslant
   \frac{4}{\tau(2\mu+5)} \|{\tn}\| \|{\tE}\|.
\end{eqnarray*}
The condition (\ref{newcond1}) then implies that
\begin{eqnarray*}
 \frac{16}{\tau^2(2\mu+5)^2}
  - 4 \left(\frac{8}{3\tau(2\mu+5)}+\frac{4\nu}{3(2\mu+5)}\right)
      \big(\mu\nu+\frac{\nu\rho}{\lambda^2}\big)
  < 0.
\end{eqnarray*}
Thus, there exist positive constants $l_i, i=1,2$ such that
\begin{eqnarray}\label{3.15}
 l_1\|{\tn}\|_1^2 + l_2\|{\tE}\|_1^2
  \leqslant &
   \frac{8}{3\tau(2\mu+5)} \|{\tE}\|^2
   + \frac{4\nu}{3(2\mu+5)} \|{\tE}_x\|^2
   + \mu\nu \|{\tn}_x\|^2  \nonumber \\
&  + \frac{\nu}{\lambda^2} \int_0^1 \rho(x){\tn}^2 {\rm d}x
   - \frac{4}{\tau(2\mu+5)} \int_0^1 {\tn} {\tE} {\rm d}x.
\end{eqnarray}
Since $\|{n^*}-\rho(x)\|_{L^{\infty}(0,1)}\leqslant C\delta_0$ for some positive constant $C$,
taking (\ref{3.12})--(\ref{3.15}) into account, we obtain
\begin{eqnarray}\label{3.16}
&&\hspace{-0.8in} \frac{\rm d}{{\rm d}t}
  \big(\frac{\mu}{2}\|{\tn}\|^2
       + \frac{\epsilon^2}{36}\|{\tn}_x\|^2
       + \frac12 \|{\tJ}\|^2
       + \frac{\lambda^2}{2}\int_0^1 {n^*}{\tV}_x^2 {\rm d}x
       + \frac{2}{3(2\mu+5)}\|{\tE}\|^2\big)
  + \frac{\epsilon^2}{18}\nu \|{\tn}_{xx}\|^2 \nonumber \\
&&\hspace{-0.7in} + l_1 \|{\tn}\|_1^2
  + l_2 \|{\tE}\|_1^2
  + \nu \|{\tJ}_x\|^2
  + \frac{1}{\tau} \|{\tJ}\|^2 \nonumber \\
  && \hspace{-1in} \leqslant
   C(\delta_0+\delta_{t^*}+J_b)
    (\epsilon^2\|{\tn}_{xx}\|^2 + \|{\tn}\|_1^2 + \|{\tJ}\|_1^2
     + \|{\tE}\|_1^2 + \|{\tV}\|_1^2 + \|{\tV}_{xt}\|^2).
\end{eqnarray}

Next, we establish  estimates for higher order derivatives of
$({\tn},{\tJ},{\tE},{\tV})$.  To this end,
 differentiate $(\ref{3.2})_1, (\ref{3.2})_2$ and
$(\ref{3.2})_3$ with respect to $x$ and test the resultant equations
by $\mu{\tn}_x, {\tJ}_x$ and $\displaystyle
\frac{4}{3(2\mu+5)}{\tE}_x$, respectively to get
\begin{eqnarray}
&&\hspace{-0.8in} \mu\int_0^1 {\tn}_{xt}{\tn}_x {\rm d}x
  - \mu\nu\int_0^1 {\tn}_{xxx}{\tn}_x {\rm d}x
  + \mu\int_0^1 {\tJ}_{xx}{\tn}_x {\rm d}x  \nonumber \\
 && \hspace{-1in} =  \frac{\mu}{2}\frac{\rm d}{{\rm d}t}\|{\tn}_x\|^2
    + \mu\nu\|{\tn}_{xx}\|^2
    + \mu\int_0^1 {\tJ}_{xx}{\tn}_x {\rm d}x
  = 0, \label{new6}
\end{eqnarray}
\begin{eqnarray}
&& \hspace{-0.8in} \int_0^1 {\tJ}_{xt}{\tJ}_x {\rm d}x
  - \nu\int_0^1 {\tJ}_{xxx}{\tJ}_x {\rm d}x
  + \frac{1}{\tau}\int_0^1 {\tJ}_x^2 {\rm d}x
  - \frac{\epsilon^2}{18}\int_0^1 {\tn}_{xxxx}{\tJ}_x {\rm d}x \nonumber  \\
&& \hspace{-0.6in} + \frac23\int_0^1 {\tE}_{xx}{\tJ}_x {\rm d}x
  - \int_0^1 ({n^*}{\tV}_x)_x {\tJ}_x {\rm d}x
  + \mu\int_0^1 {\tn}_{xx} {\tJ}_x {\rm d}x  \nonumber \\
  && \hspace{-1in}=  \frac12\frac{\rm d}{{\rm d}t} \|{\tJ}_x\|^2
    + \nu \|{\tJ}_{xx}\|^2
    + \frac{1}{\tau} \|{\tJ}_x\|^2
    - \frac{\epsilon^2}{18}
      \int_0^1 {\tn}_{xxxx}(\nu{\tn}_{xx}-{\tn}_t) {\rm d}x
    + \frac23\int_0^1 {\tE}_{xx}{\tJ}_x {\rm d}x \nonumber  \\
&&  \hspace{-0.8in}  - \int_0^1 n^*_{x} {\tV}_x {\tJ}_x {\rm d}x
    - \int_0^1 {n^*} {\tV}_{xx} {\tJ}_x {\rm d}x
    - \mu \int_0^1 {\tn}_x {\tJ}_{xx} {\rm d}x \nonumber  \\
  && \hspace{-1in}= \frac12\frac{\rm d}{{\rm d}t}
    \big(\|{\tJ}_x\|^2+\frac{\epsilon^2}{18}\|{\tn}_{xx}\|^2\big)
    + \nu \|{\tJ}_{xx}\|^2
    + \frac{1}{\tau} \|{\tJ}_x\|^2
    + \frac{\epsilon^2}{18}\nu \|{\tn}_{xxx}\|^2
    + \frac23\int_0^1 {\tE}_{xx} {\tJ}_x {\rm d}x\nonumber   \\
&&   \hspace{-0.8in} - \int_0^1 n^*_{x} {\tV}_x {\tJ}_x {\rm d}x
    - \frac{1}{\lambda^2} \int_0^1 {n^*} {\tn}(\nu{\tn}_{xx}-{\tn}_t) {\rm d}x
    - \mu\int_0^1 {\tn}_x {\tJ}_{xx} {\rm d}x  \nonumber \\
  && \hspace{-1in} =\frac12\frac{\rm d}{{\rm d}t}
    \big( \|{\tJ}_x\|^2+\frac{\epsilon^2}{18}\|{\tn}_{xx}\|^2
          +\frac{1}{\lambda^2}\int_0^1 {n^*}{\tn}^2 {\rm d}x \big)
    + \nu \|{\tJ}_{xx}\|^2
    + \frac{1}{\tau} \|{\tJ}_x\|^2
    + \frac{\epsilon^2}{18}\nu \|{\tn}_{xxx}\|^2  \nonumber \\
&&  \hspace{-0.8in}  + \frac{\nu}{\lambda^2} \int_0^1 {n^*}{\tn}_x^2 {\rm d}x
    + \frac{\nu}{\lambda^2} \int_0^1 n^*_{x}{\tn}{\tn}_x {\rm d}x
    + \frac23\int_0^1 {\tE}_{xx} {\tJ}_x {\rm d}x
    - \int_0^1 n^*_{x} {\tV}_x {\tJ}_x {\rm d}x
    - \mu\int_0^1 {\tn}_x {\tJ}_{xx} {\rm d}x  \nonumber  \\
  &&  \hspace{-1in}=\int_0^1 g_{1x} {\tJ}_x {\rm d}x, \label{new7}
\end{eqnarray}
and
\begin{eqnarray}
&& \hspace{-0.8in} \frac{2}{3(2\mu+5)}\frac{\rm d}{{\rm d}t} \|{\tE}_x\|^2
    + \frac{4}{3(2\mu+5)}
      ( \nu\|{\tE}_{xx}\|^2 + \frac{2}{\tau} \|{\tE}_x\|^2)
    - \frac{4}{\tau(2\mu+5)} \int_0^1 {\tn}_x {\tE}_x {\rm d}x
      \nonumber  \\
   &&  \hspace{-1in}= \frac23 \int_0^1 {\tJ}_x {\tE}_{xx} {\rm d}x + \frac{4}{3(2\mu+5)} \int_0^1 g_{2x} {\tE}_x {\rm d}x. \label{new8}
\end{eqnarray}

Adding equations (\ref{new6})--(\ref{new8}) results in
\begin{eqnarray}
&&  \hspace{-0.8in} \frac12 \frac{\rm d}{{\rm d}t}
  \big( \mu\|{\tn}_x\|^2+\frac{\epsilon^2}{18}\|{\tn}_{xx}\|^2
        + \|{\tJ}_x\|^2 + \frac{1}{\lambda^2}\int_0^1 {n^*}{\tn}^2 {\rm d}x
        + \frac{4}{3(2\mu+5)}\|{\tE}_x\|^2 \big) \nonumber \\
&& \hspace{-0.8in}  + \mu\nu \|{\tn}_{xx}\|^2
  + \frac{\epsilon^2}{18}\nu \|{\tn}_{xxx}\|^2
  + \frac{\nu}{\lambda^2} \int_0^1 {n^*} {\tn}_x^2 {\rm d}x
  + \nu \|{\tJ}_{xx}\|^2
  + \frac{1}{\tau} \|{\tJ}_x\|^2
  + \frac{4\nu}{3(2\mu+5)} \|{\tE}_{xx}\|^2 \nonumber \\
&& \hspace{-0.8in}  + \frac{8}{3\tau(2\mu+5)} \|{\tE}_x\|^2
  - \frac{4}{\tau(2\mu+5)} \int_0^1 {\tn}_x {\tE}_x {\rm d}x \nonumber \\
  && \hspace{-1in}=  \int_0^1 g_{1x} {\tJ}_x {\rm d}x
    + \frac{4}{3(2\mu+5)} \int_0^1 g_{2x} {\tE}_x {\rm d}x
    - \frac{\nu}{\lambda^2} \int_0^1 n^*_{x}{\tn}{\tn}_x {\rm d}x
    + \int_0^1 n^*_{x} {\tV}_x {\tJ}_x {\rm d}x.  \label{3.17}
\end{eqnarray}
Since
\begin{eqnarray*}
 \left\{
  \begin{array}{ll}
   \displaystyle \|{\tn}_x\|^2
    = \int_0^1 \big(\int_0^x {\tn}_{yy}(y) {\rm d}y\big)^2 {\rm d}x
    \leqslant
      \int_0^1 {\tn}_{yy}^2 {\rm d}y
    = \|{\tn}_{xx}\|^2, \medskip \\
   \displaystyle \|{\tE}_x\|^2
    = \int_0^1 \big(\int_0^x {\tE}_{yy}(y) {\rm d}y\big)^2 {\rm d}x
    \leqslant
      \int_0^1 {\tE}_{yy}^2 {\rm d}y
    = \|{\tE}_{xx}\|^2,
  \end{array}
 \right.
\end{eqnarray*}
and
\begin{eqnarray*}
 \left| \frac{4}{\tau(2\mu+5)} \int_0^1 {\tn}_x {\tE}_x {\rm d}x \right|
  \leqslant
   \frac{4}{\tau(2\mu+5)} \|{\tn}_x\| \|{\tE}_x\|.
\end{eqnarray*}
Then similar to (\ref{3.15}),  under the condition (\ref{newcond1})
it holds that
\begin{eqnarray}\label{3.18}
&& \hspace{-0.8in} l_1 \|{\tn}_x\|_1^2 + l_2 \|{\tE}_x\|_1^2
  \leqslant
   \mu\nu \|{\tn}_{xx}\|^2
    + \frac{\nu}{\lambda^2} \int_0^1 \rho(x) {\tn}_x^2 {\rm d}x
    + \frac{4\nu}{3(2\mu+5)} \|{\tE}_{xx}\|^2
  \nonumber \\
&&   \hspace{1in}  + \frac{8}{3\tau(2\mu+5)} \|{\tE}_x\|^2 - \frac{4}{\tau(2\mu+5)} \int_0^1 {\tn}_x {\tE}_x {\rm d}x.
\end{eqnarray}
In addition, it is straightforward to check that
\begin{eqnarray}\label{3.19}
&&\hspace{-0.8in} \int_0^1 g_{1x} {\tJ}_x {\rm d}x
  + \frac{4}{3(2\mu+5)} \int_0^1 g_{2x} {\tE}_x {\rm d}x
  - \frac{\nu}{\lambda^2} \int_0^1 n^*_{x} {\tn} {\tn}_x {\rm d}x
 + \int_0^1 n^*_{x} {\tV}_x {\tJ}_x {\rm d}x  \nonumber \\
&& \hspace{-1in} \leqslant
  C(\delta_0+\delta_{t^*}+J_b)(\|{\tn}\|_2^2 + \epsilon^2\|{\tn}_{xxx}\|^2
                           + \|{\tJ}\|_2^2 + \|{\tE}\|_2^2 + \|{\tV}\|_2^2).
\end{eqnarray}
Combining (\ref{3.17})--(\ref{3.19}) leads to
\begin{eqnarray}\label{3.20}
&&\hspace{-0.5in} \frac12 \frac{\rm d}{{\rm d}t}
  \big( \mu\|{\tn}_x\|^2 + \frac{\epsilon^2}{18}\|{\tn}_{xx}\|^2 +
        \|{\tJ}_x\|^2 + \frac{1}{\lambda^2}\int_0^1 {n^*}{\tn}^2 {\rm d}x
        + \frac{4}{3(2\mu+5)}\|{\tE}_x\|^2 \big) \nonumber \\
&& \hspace{-0.4in}+ l_1\|{\tn}_x\|_1^2
  + l_2\|{\tE}_x\|_1^2
  + \nu\|{\tJ}_{xx}\|^2
  + \frac{1}{\tau}\|{\tJ}_x\|^2
  + \frac{\epsilon^2}{18}\nu\|{\tn}_{xxx}\|^2 \nonumber \\
&&\hspace{-0.7in}  \leqslant
    C(\delta_0+\delta_{t^*}+J_b)
     (\|{\tJ}\|_2^2 + \|{\tn}\|_2^2
      + \epsilon^2\|{\tn}_{xxx}\|^2 + \|{\tV}\|_2^2 + \|{\tE}\|_2^2).
\end{eqnarray}

Following the same procedure, we can also establish the   estimate for higher order derivatives
of $({\tn}, {\tJ}, {\tE}, {\tV})$ as follows
\begin{eqnarray}\label{3.21}
&&\hspace{-0.8in} \frac12 \frac{\rm d}{{\rm d}t}
  \left( \mu\|{\tn}_{xx}\|^2 + \frac{\epsilon^2}{18}\|{\tn}_{xxx}\|^2
        + \|{\tJ}_{xx}\|^2 + \frac{1}{\lambda^2}\int_0^1 {n^*}{\tn}_x^2 {\rm d}x
        + \frac{4}{3(2\mu+5)}\|{\tE}_{xx}\|^2 \right) \nonumber \\
&&\hspace{-0.6in} + l_1\|{\tn}_{xx}\|_1^2
  + l_2\|{\tE}_{xx}\|_1^2
  + \nu\|{\tJ}_{xxx}\|^2
  + \frac{1}{\tau}\|{\tJ}_{xx}\|^2
  + \frac{\epsilon^2}{18}\nu\|{\tn}_{xxxx}\|^2 \nonumber \\
&&\hspace{-1in}  \leqslant
    C(\delta_0+\delta_{t^*}+J_b) (\|{\tJ}\|_3^2 + \|{\tn}\|_3^2
      + \epsilon^2\|{\tn}_{xxxx}\|^2 + \|{\tV}\|_3^2 + \|{\tE}\|_3^2).
\end{eqnarray}

It follows from $(\ref{3.2})_1$ that
\begin{eqnarray}\label{3.22}
 \|\partial_x^k {\tJ}\|^2
& = \int_0^1
    (\nu\partial_x^{k+1}{\tn}-\partial_x^{k-1}{\tn}_t)^2 {\rm d}x \nonumber \\
 & = \int_0^1\big[ \nu^2(\partial_x^{k+1}{\tn})^2
                   + (\partial_x^{k-1}{\tn}_t)^2
                   - 2(\nu\partial_x^{k+1}{\tn})(\partial_x^{k-1}{\tn}_t)
                   \big] {\rm d}x \nonumber \\
& = \nu^2\|\partial_x^{k+1} {\tn}\|^2
    + \|\partial_x^{k-1} {\tn}_t\|^2
    + \nu\frac{\rm d}{{\rm d}t} \|\partial_x^k {\tn}\|^2, \ \
    k=1,2,3,\cdots
\end{eqnarray}
In addition, observe that
\begin{eqnarray}\label{3.7}
\hspace{-0.8in}\left\{
 \begin{array}{ll}
  \displaystyle
  \|{\tV}\|^2
   = \int_0^1\Big(\int_0^x {\tV}_y {\rm d}y\Big)^2 {\rm d}x
   \leqslant
    \|{\tV}_x\|^2
     = \int_0^1\Big(\int_0^x {\tV}_{yy} {\rm d}y\Big)^2 {\rm d}x
     \leqslant
      \|{\tV}_{xx}\|^2
       = \frac14 \|{\tn}\|^2,  \medskip \\
  \displaystyle
  \|{\tV}_{xt}\|^2
   = \int_0^1\Big(\int_0^x {\tV}_{yyt} {\rm d}y\Big)^2 {\rm d}x
   \leqslant
    \|{\tV}_{xxt}\|^2
     = \frac14 \|{\tn}_t\|^2.
 \end{array}
\right.
\end{eqnarray}
Collecting (\ref{3.16}), (\ref{3.20}) and (\ref{3.21}) then gives
\begin{eqnarray*}
& \frac12 \frac{\rm d}{{\rm d}t}
  \big( \mu\|{\tn}\|_2^2 + (\frac{\epsilon^2}{18}+\nu^2)\|{\tn}_x\|_2^2
        + \|{\tJ}\|_2^2 + \frac{4}{3(2\mu+5)}\|{\tE}\|_2^2\\
      &  + \lambda^2 \int_0^1 {n^*}({\tV}_x^2+{\tV}_{xx}^2+{\tV}_{xxx}^2) {\rm d}x
        \big)  \\
& + l_1\|{\tn}\|_3^2
  + l_2\|{\tE}\|_3^2
  + \frac{\nu}{2}(\nu^2+\frac{\epsilon^2}{9})\|{\tn}_{xx}\|_2^2
  + \frac{\nu}{3}\|{\tn}_t\|_2^2
  + \frac{\nu}{2}\|{\tJ}_x\|_2^2
  + \frac{1}{\tau}\|{\tJ}\|_2^2  \\
  \leqslant &
    C(\delta_0+\delta_{t^*}+J_b) (\|{\tn}\|_3^2+\epsilon^2\|{\tn}_{xx}\|_2^2
      + \|{\tE}\|_3^2 + \|{\tV}\|_3^2).
\end{eqnarray*}
Applying Gronwall's inequality to the above inequality and provided
that $\delta_0+\delta_{t^*}+J_b$ is sufficiently small we conclude  there
exist two positive constants $\varsigma_1$ and $\varsigma_2$ such
that
\begin{eqnarray*}
& \|{\tn}(t)\|_3^2
  + \|{\tJ}(t)\|_2^2
  + \|{\tE}(t)\|_2^2
  + \|{\tV}(t)\|_5^2
   \\ &+ \varsigma_2 \int_0^t e^{\varsigma_1 (s-t)}
    (\|{\tn}(s)\|_4^2 + \|{\tJ}(s)\|_3^2 + \|{\tE}(s)\|_3^2 + \|{\tn}_s(s)\|_2^2)
    {\rm d}s  \\
 \leqslant &
    e^{-\varsigma_1 t}(\|{\tn}_0\|_3^2 + \|{\tJ}_0\|_2^2 + \|{\tE}_0\|_2^2
    + \|{\tV}_0\|_3^2),
\end{eqnarray*}
which gives (\ref{3.5}). The proof is thus complete.
\end{proof}

$\mathbf{Proof \ of \ Theorem\, \ref{T1.2}:}$ \ First notice that
$\delta_{t^*}$ defined by (\ref{3.4}) is sufficiently small as long
as the initial data
$\|{\tn}_0\|_3+\|{\tJ}_0\|_2^2+\|{\tE}_0\|_2^2+\|{\tV}_0\|_3^2$ is
small enough.

Then, based on Lemma \ref{L3.1}, it follows directly from
(\ref{3.1}) that
\begin{eqnarray*}
& \|n(t)-{n^*}\|_3^2
  + \|J(t)-{J^*}\|_2^2
  + \|\mathcal{E}(t)-\mathcal{E}^*\|_2^2
  + \|V(t)-V^*\|_5^2 \\
 \leqslant &
  e^{-\varsigma_1 t} (\|{\tn}_0\|_3^2+\|{\tJ}_0\|_2^2+\|{\tE}_0\|_2^2+\|{\tV}_0\|_3^2),
\end{eqnarray*}
which further implies that $(n, J, \mathcal{E}, V)$ exists globally
in time and satisfies (\ref{1.9}).
\section{Acknowledgment} \label{sec:ack}

This work is supported by the Junta de Andaluc\'ia and FEDER, Spain
(No. P18-FR-4509); the Simons Foundation, USA (Collaboration Grants
for Mathematicians No. 429717); the National Nature Science
Foundation of China (No. 12171082); the fundamental research funds
for the central universities (No. 2232022G -13, 2232023G -13); the
Natural Science Foundation of Hubei Province, China (No. 2022CFB661)
and the Young and Middle-aged Talent Fund of Hubei Education
Department, China (No. Q20201307).


\begin{appendix}
\section{List of parameters}\label{app:para}

\begin{eqnarray*}
a_2 &= & - \frac{1}{\tau} J_b
      - (1+\mu)\rho_x(x)
      + \frac23\frac{J_b^2 \rho_x}{\rho^2(x)}
      + \frac{\epsilon^2}{18} \frac{\rho_x^3(x)}{\rho^2(x)},
      \\
b_2 &=& \frac{1}{\tau}
      - \frac43\frac{J_b \rho_x}{\rho^2(x)},  \\
c_2 &= & \frac43 \frac{J_b^2 \rho_x^2}{\rho^3(x)}
      - \frac{\epsilon^2}{9}\frac{\rho_x^4(x)}{\rho^3(x)},
      \\
d_2 &=& \mu - \frac23 \frac{J_b^2}{\rho^2(x)}
      - \frac{\epsilon^2}{6}\frac{\rho_x^2(x)}{\rho^2(x)},
      \\
e_2 &=& \frac43\frac{J_b}{\rho(x)}
      - \frac{\epsilon^2}{9\nu}\frac{\rho_x(x)}{\rho(x)},  \\
f_2 &= & p V^*_{x}
      - \frac23(\frac{(J^*)^2}{n^*})_x
      + \frac{\epsilon^2}{18} \frac{(n^*_{x})^3}{(n^*)^2}
      + \frac{\epsilon^2}{9\nu} \frac{n^*_{x}}{n^*} J^*_{x}
      + \frac43 \frac{J_b}{\rho(x)}q_x \\
     && - \frac23 \left( \frac{J_b^2\rho_x}{\rho^2(x)}
        - \frac{2J_b^2\rho_x^2}{\rho^3(x)}p
        + \frac{J_b^2}{\rho^2}p_x
        + \frac{2J_b\rho_x}{\rho^2(x)}q \right) \\
    && - \frac{\epsilon^2}{18} \left(
        \frac{3\rho_x^2}{\rho^2(x)}p_x
        + \frac{2\rho_x^4}{\rho^3(x)}p
        + \frac{\rho_x^3(x)}{\rho^2(x)} \right)
      - \frac{\epsilon^2}{9\nu}\frac{\rho_x}{\rho(x)}q_x \\
    &= & \mathcal{O}(p^2 + q^2 + p_x^2 + q_x^2 + (V^*_x)^2), \\
a_3 &= &- \frac23 \frac{J_b^3}{\rho^3(x)}
      + \frac{\epsilon^2}{9} \frac{J_b\rho_x}{\rho^3(x)},
      \\
b_3 &=& \frac{2}{\tau}
      - \frac53 \frac{J_b \rho_x}{\rho^2(x)}, \\
c_3 &= &- \frac{3}{\tau}
      + \frac52 \frac{J_b\rho_x^2}{\rho^2(x)}
      - \frac{2J_b^3 \rho_x}{\rho^4(x)}
      + \frac{\epsilon^2}{3} \frac{J_b \rho_x^4}{\rho^4(x)},
      \\
d_3 &=& - \frac{5J_b}{2\rho(x)}
      - \frac{\epsilon^2}{3} \frac{J_b\rho_x^2}{\rho^3(x)},
      \\
e_3 &=&\mu + \frac{5}{2}
      - \frac{J_b^2}{\rho^2(x)}
      + \frac{\epsilon^2}{6\nu} \frac{J_b\rho_x}{\rho^2(x)}
      + \frac{\epsilon^2}{18} \frac{\rho_x^2}{\rho^2(x)},
      \\
h_3 &=& \frac{2J_b^2}{\rho^3(x)}
      - \frac{\epsilon^2}{9}\frac{\rho_x^3(x)}{\rho^3(x)},\\
f_3 &= & q V^*_{x} - \frac53 \left(\frac{J^*}{n^*}\mathcal{E}^* \right)_x
      + \frac13 \left(\frac{(J^*)^3}{(n^*)^2}\right)_x
      + \frac{\epsilon^2}{18}
        \left( \frac{J^*n^*_{xx}}{n^*}-\frac{J^*(n^*_{x})^2}{(n^*)^2} \right)_x
    \\
    &&   - \frac{\epsilon^2}{18\nu} \frac{J^*J^*_{xx}}{n^*} + \frac53 \big(\frac32\frac{J_b\rho_x}{\rho(x)}+\frac32 q_x
                 + \frac{J_b}{\rho(x)}r_x
                 + \frac32 \frac{\rho_x}{\rho(x)}q
                 - \frac32 \frac{J_b\rho_x^2}{\rho^2(x)}p \big)\\
    && - \frac53 \left( \frac32 \frac{J_b\rho_x}{\rho(x)}
                 - \frac{3J_b\rho_x^2(x)}{\rho^2(x)}p
                 + \frac{J_b\rho_x}{\rho^2(x)}r
                 + \frac32 \frac{\rho_x(x)}{\rho(x)}q
                 + \frac32 \frac{J_b}{\rho(x)}p_x \right)
      - \frac{J_b^2}{\rho^2(x)}q_x  \\
    && + \frac23 \left( \frac{J_b^3}{\rho^3(x)}
                 - \frac{3J_b^3\rho_x}{\rho^4(x)}p
                 + \frac{3J_b^2}{\rho^3(x)}q \right)
      + \frac{\epsilon^2}{6\nu}
        \frac{J_b\rho_x}{\rho^2(x)}q_x
      + \frac{\epsilon^2}{18}
        \frac{\rho_x^2(x)}{\rho^2(x)}q_x \\
    && - \frac{\epsilon^2}{9}\left( \frac{J_b\rho_x^3}{\rho^3(x)}
                             - \frac{3J_b\rho_x^4}{\rho^4(x)}p
                             + \frac{3J_b\rho_x^2}{\rho^3(x)}p_x
                             + \frac{\rho_x^3(x)}{\rho^3(x)}q \right)\\
  & = & \mathcal{O}(p^2+q^2+r^2+(V^*_x)^2+p_x^2+q_x^2+r_x^2),\\
\tilde{f}_3
&=&  \frac{u_{2xx}^{(1)}(u_2^{(1)}+J_b)}{u_1^{(1)}+\rho(x)}
  -\frac{\epsilon^2}{18\nu}
     \frac{u_{2xx}^{(2)}(u_2^{(2)}+J_b)}{u_1^{(2)}+\rho(x)}\\
 g_1 &= &
       - \frac{\epsilon^2}{18} \big[\frac{2(n^*+\tn)_x(n^*+\tn)_{xx}}{n^*+\tn}
         - \frac{2n^*_{x}n^*_{xx}}{n^*} - \frac{(n^*+\tn)_x^3}{n^*+\tn}
         + \frac{(n^*_{x})^3}{n^*} \big] \\
       && - \frac23 \big[\frac{(J^*+\tJ)^2}{n^*+\rho}-\frac{(J^*)^2}{n^*}\big]_x + V^*_{x}\rho
       + \rho \tV_x,  \\
 g_2 &= & - \frac53 \big[\frac{J^*+\tJ}{{n^*}+\tn}(\mathcal{E}^*+\tE)
                        -\frac{J^*}{n^*}\mathcal{E}^*
                        -\frac32 \tJ \big]_x\\
    &&   + \frac13\big[\frac{(J^*+\tJ)^3}{(n^*+\tn)^2}-\frac{(J^*)^3}{(n^*)^2}\big]_x
       + V^*_{x} \tJ
       + \tJ \tV_x
       + J^* \tV_x  \\
       &&+ \frac{\epsilon^2}{18}
         \big[\frac{(J^*+ \tJ)(n^*+\tn)_{xx}}{n^*+\tn}-\frac{J^*n^*_{xx}}{n^*}
              - \frac{(J^*+\tJ)(n^*+\tn)_x^2}{(n^*+\tn)^2}
              + \frac{J^* (n^*_{x})^2}{(n^*)^2}\big]_x, \\
               g_{\sigma 1} &= &
       - \frac{\epsilon^2}{18} \big[
         \frac{\big(({n^*}+\tn_{\sigma})_x^2\big)_x}{{n^*}+\tn_{\sigma}}
         - \frac{2n^*_{x}n^*_{xx}}{{n^*}}
         - \frac{({n^*}+\tn_{\sigma})_x^3}{{n^*}+\tn_{\sigma}}
         + \frac{(n^*_{x})^3}{{n^*}} \big] \\
       &&- \frac23 \big[
         \frac{({J^*}+\tJ_{\sigma})^2}{{n^*}+\tn_{\sigma}}-\frac{({J^*})^2}{{n^*}}\big]_x + V^*_{x}\tn_{\sigma}
       + \tn_{\sigma} \tV_{\sigma x},  \\
 g_{\sigma 2}
 & = & - \frac53 \big[\frac{{J^*}+\tJ_{\sigma}}{{n^*}+\tn_{\sigma}}(\mathcal{E}^*+\tE_{\sigma})
                     -\frac{{J^*}}{{n^*}}\mathcal{E}^*\big]_x
       + \frac13\big[
         \frac{({J^*}+\tJ_{\sigma})^3}{({n^*}+\tn_{\sigma})^2}-\frac{{J^*}^3}{{n^*}^2}\big]_x
      \\ && + V^*_{x}\tJ_{\sigma}
       + \tJ_{\sigma} \tV_{\sigma x}
       + {J^*} \tV_{\sigma x}
       + \frac52 \tJ_{\sigma x} \\
       &&+ \frac{\epsilon^2}{18}
         \big[\frac{({J^*}+\tJ_{\sigma})({n^*}+\tn_{\sigma})_{xx}}{{n^*}+\tn_{\sigma}}
              -\frac{{J^*}n^*_{xx}}{{n^*}}
              - \frac{({J^*}+\tJ_{\sigma})({n^*}+\tn_{\sigma})_x^2}{({n^*}+\tn_{\sigma})^2}
              + \frac{J^* (n^*_{x})^2}{{n^*}^2}\big]_x , \medskip \\
 \hat{g}_1 &= & - \frac23 \big[ \frac{(J^*+ \tJ_1)^2}{n^*+\tn_1}
                 -\frac{(J^*+\tJ_2)^2}{n^*+\tn_2} \big]_x
                 + V^*_{x}(\tn_1 - \tn_2)
                 + \tn_1 \tV_{1x}
                 - \tn_2 \tV_{2x}  \\
      &&  - \frac{\epsilon^2}{18}
           \big[\frac{\big((n^*+\tn_1)_x^2\big)_x}{n^*+\tn_1}
         - \frac{\big((n^*+\tn_2)_x^2\big)_x}{n^*+\tn_2}
         - \frac{(n^*+\tn_1)_x^3}{n^*+\tn_1}
         + \frac{(n^*+\tn_2)_x^3}{n^*+\tn_2} \big]  \medskip\\
 \hat{g}_2
  &= & - \frac53 \big[\frac{J^*+\tJ_1}{n^*+\tn_1}(\mathcal{E}^*+\tE_1)
      - \frac{J^*+\tJ_2}{n^*+\tn_2}(\mathcal{E}^*+\tE_2)\big]_x
      + \frac13 \big[ \frac{(J^*+\tJ_1)^3}{(n^*+\tn_1)^2}
                     - \frac{(J^*+\tJ_2)^3}{(n^*+\tn_2)^2} \big]_x \\
    && + \frac{\epsilon^2}{18}
         \big[\frac{(J^*+\tJ_1)(n^*+\tn_1)_{xx}}{n^*+\tn_1}
              - \frac{(J^*+\tJ_2)(n^*+\tn_2)_{xx}}{n^*+\tn_2} \\
    && \qquad\quad
              - \frac{(J^*+\tJ_1)(n^*+\tn_1)_x^2}{(n^*+\tn_1)^2}
              + \frac{(J^*+\tJ_2)(n^*+\tn_2)_x^2}{(n^*+\tn_2)^2}
              \big]_x \\
    && + V^*_{x}(\tJ_1 - \tJ_2)
      + \tJ_1 \tV_{1x} - \tJ_2 \tV_{2x}
      + J^*(\tV_{1x} - \tV_{2x}).
\end{eqnarray*}
\end{appendix}


\section*{References}
\begin{thebibliography}{99}


\bibitem{AT87} M. G. Ancona, H. F. Tiersten, Microscopic physics
of the sillicon inversion layer, {\it Phys. Revi. B}, 35 (1987),
7959-7965.

\bibitem{AI89} M. G. Ancona, G. I. Iafrate, Quantum correction to
the equation of state of an electron gas in a semiconductor, {\it
Phys. Revi. B}, 39 (1989), 9536-9540.

\bibitem{AM09} P. Antonelli, P. Marcati, On the finite energy weak
solutions to a system in quantum fluid dynamics, {\it Comm. Math.
Phys.}, 39287 (2009), 657-686.

\bibitem{AM12} P. Antonelli, P. Marcati, The quantum hydrodynamics
system in two space dimenisions, {\it Arch. Ration. Mech. Anal.},
203 (2012), 499-527.

\bibitem{B52} D. Bohm, A suggested interpretation of the quantum
theory in terms of ``hidden'' valuables: I; II, {\it Phys. Revi.},
85 (1952), 166-179.



\bibitem{CD07} L. Chen, M. Dreher, The viscous model of quantum
hydrodynamics in several dimensions, {\it Math. Models Methods Appl.
Sci.}, 17 (2007), 1065-1093.

\bibitem{CD11} L. Chen, M. Dreher, Viscous quantum hydrodynamics and
parameter-elliptic systems, {\it Math. Methods Appl. Sci.}, 34
(2011), 520-531.


\bibitem{D08} M. Dreher, The transient equations of viscous quantum
hydrodynamics, {\it Math. Models Methods Appl. Sci.}, 31 (2008),
391-414.


\bibitem{F72} R. Feynman, {\it Statistical mechanics, a set of
lecture}, New York: W. A. Benjamin, 1972.




\bibitem{GLM00} I. Gasser, C. C. Lin, P. Markowich, A review of
dispersive limits of the (non)linear Schr${\rm \ddot{o}}$dinger-type
eqnarray, {\it Taiwanese J. of Math.}, 4 (2000), 501-529.


\bibitem{GJ01} I. Gamba, A. J${\rm \ddot{u}}$ngel, Positive solutions to
singular second and third order differential equations for quantum
fluids, {\it Arch. Ration. Mech. Anal.}, 156 (2001), 183-203.

\bibitem{GJT03} M. Gualdani, A. J${\rm\ddot{u}}$ngel, and G.
Toscani, Exponential decay in time of solutons of the viscous
quantum hydrodynamic equations, {\it Appl. Math. Lett.}, 16 (2003),
1273-1278.

\bibitem{GJV09} I. M. Gamba, A. J${\rm \ddot{u}}$ngel, A. Vasseur,
Global existence of solutions to one-dimensional viscous quantum
hydrodynamic eqnarrays, {\it J. Differential Equations}, 247 (2009),
3117-3135.

\bibitem{HJL03} C. C. Hao, Y. L. Jia, H. L. Li, Quantum
Euler-Poisson system: Local existence of solution, {\it J. Partial
Differential Equations}, 16 (2003), 306-320.

\bibitem{HLM06} F. Huang, H. L. Li, A. Matsumura, Existence and
stability of steady-state of one-dimensional quantum hydrodynamic
system for semiconductors, {\it J. Differential Equations}, 225
(2006), 1-25.

\bibitem{HLMO10} F. Huang, H. L. Li, A. Matsumura, S. Odanaka,
Well-posedness and stability of multidimensional quantum
hydrodynamics in $\mathbb{R}^3$, {\it Series in Contemporary Applied
Mathematics CAM 15}, High Education Press, Beijing, 2010.

\bibitem{HZ} H. Hu, K. Zhang, Stability and semi-classical limit in a semiconductor
full quantum hydrodynamic model with non-flat doping profile, {\it
arXiv preprint arXiv: 1706.01621}, (2017).



\bibitem{J01} A. J${\rm \ddot{u}}$ngel, {\it Quasi-hydrodynamic Semiconductor
eqnarrays, Progr. Nonlinear Differential Equations}, Birkh${\rm
\ddot{a}}$user, Basel, 2001.


\bibitem{JL04} A. J${\rm \ddot{u}}$ngel, H. L. Li, Quantum
Euler-Poisson systems: global existence and exponential decay, {\it
Quarterly of Applied Mathematics}, 62 (2004), 269-600.




\bibitem{JLM06} A. J${\rm \ddot{u}}$ngel, H. L. Li, A. Matsumura,
the relaxation-time limit in the quantum hydrodynamic equations for
semiconductor, {\it J. Differential eqnarrays}, 225 (2006), 440-464.

\bibitem{JM07} A. J${\rm\ddot{u}}$ngel, J. Milisi${\rm\acute{c}}$,
Physical and numerical viscosity for quantum hydrodynamics, {\it
Commun. Math. Phys.}, 5 (2007), 447-471.



\bibitem{LM93} M. Loffredo, L. Morato, On the creation of quantum
vortex lines in rotating HeII, {\it IL NUOVO CIMENTO}, 108B (1993),
205-215.


\bibitem{LM04} H. L. Li, P. Marcati, Existence and asymptotic
behavior of multidimensional quantum hydrodynamic model for
semiconductors, {\it Commun. Math. Phys.}, 245 (2004), 215-247.

\bibitem{LL05} H. L. Li, C. C. Lin, Zero-Debye length asymptotic of
the quantum hydrodynamic model for semiconductors, {\it Commun.
Math. Phys.}, 256 (2005), 159-212.

\bibitem{MRS90} P. A. Markowich, C. Ringhofer, C. Schmeiser,
{\it Semiconductor eqnarrays}, Springer, Wien, 1990.


\bibitem{SLH23} W. Sun, Y. Li, X. Han, The full viscous quantum
hydrodynamic system in one dimensional space, {\it J. Math. Phys.},
64 (2023), 011501-1-23.

\bibitem{RH21} S. Ra, H. Hong, The existence, uniqueness and
exponential decay of global solutions in the full quantum
hydrodynamic equations for semiconductors, {\it Zeitschrift
f$\ddot{u}$r angewandte Mathematik und Physik}, 72(3) (2021), 1-32.

\bibitem{T88} R. Temam, {\it Infinite-dimensional Dynamical System
in Mechanics and Physics}, Springer, New York, 1988.


\bibitem{W32} E. Wigner, On the quantum correction for thermodynamic
equilibrium, {\it Phys. Rev.}, 40 (1932), 749-759.


 \end{thebibliography}
\end{document}